\documentclass[letterpaper, 10 pt, conference]{ieeeconf/ieeeconf}  

\IEEEoverridecommandlockouts                              

\title{\LARGE \bf
Robust Planning with Quadratically Constrained State-Dependent Uncertainties
}

\author{Fabio Spada$^{1}$ and Beh\c cet A\c c\i kme\c se$^{2}$
\thanks{$^{1}$Fabio Spada is with the Department of Mechanical Engineering, University of California at Berkeley, Berkeley, CA 94720 USA (e-mail:{\tt\small fabio\_s@berkeley.edu}) \newline
$^{2}$Beh\c cet A\c c\i kme\c se is with the Department of Electrical Engineering and Computer Sciences and with the Department of Mechanical Engineering, University of California at Berkeley, Berkeley, CA 94720 USA (e-mail:{\tt\small behcet@berkeley.edu})}%
}

\usepackage{amsmath, amssymb, amsfonts} 
\usepackage{mathtools}
\usepackage{graphicx}
\usepackage{float}
\usepackage{tikz}
\usepackage{wrapfig}
\usepackage{cuted}
\usepackage{algorithm}
\usepackage{algpseudocode}
\usepackage[caption=false,font=footnotesize]{subfig}
\usepackage{booktabs}
\usepackage{multirow}
\usepackage{textcomp}
\usepackage{xcolor}
\usepackage[most]{tcolorbox}
\usepackage{adjustbox}
\usepackage{hyperref}
\usepackage{bm}
\usepackage{todonotes}

\newcommand{\calP}{\mathcal{P}}

\newcommand{\Se}{\bm{S}}
\newcommand{\lidx}[2]{{\vphantom{#1}}^{#2}\mkern-2mu#1}

\makeatletter
\newsavebox{\fixidxbox}
\DeclareRobustCommand{\fixidx}[1]{\mathpalette\fixidx@{#1}}
\newcommand{\fixidx@}[2]{%
  \sbox{\fixidxbox}{$\m@th#1#2$}%
  \ooalign{%
    \usebox{\fixidxbox}\cr
    \hidewidth\raisebox{-.22ex}{\rule{\wd\fixidxbox}{.35pt}}\hidewidth\cr
  }%
}
\DeclareRobustCommand{\luoG}{\mathpalette\luoG@{}}
\newcommand{\luoG@}[2]{%
  \ooalign{%
    $\m@th#1\underline{\overline{G}}$\cr
    \kern-.45pt\rule[-.36ex]{.45pt}{2.24ex}\cr
  }%
}
\DeclareRobustCommand{\luoN}{\mathpalette\luoN@{}}
\newcommand{\luoN@}[2]{%
  \ooalign{%
    $\m@th#1\underline{\overline{N}}$\cr
    \kern-.45pt\rule[-.36ex]{.45pt}{2.24ex}\cr
  }%
}
\makeatother

\newcounter{example}
{%
  \refstepcounter{example}%
  \par\medskip\noindent
  \begingroup
  \textbf{Example \theexample}%
    \if\relax\detokenize{#1}\relax.\else\ (#1).\fi
  \ %
  \itshape
}%
{%
  \par
  \endgroup
  \medskip
}

\newcounter{problem}
\newcommand{\problemframecolor}{black}
\makeatletter
\newtagform{blank}[\@gobble]{}{}
\makeatother
\newenvironment{problem}[1][]%
{%
  \refstepcounter{problem}%
  \begingroup
  \usetagform{blank}%
  \begin{tcolorbox}[
    breakable,
    enhanced,
    colback=white,
    colframe=\problemframecolor,
    boxrule=0.5pt,
    arc=0pt,
    outer arc=0pt,
    left=4pt,
    right=4pt,
    top=4pt,
    bottom=4pt,
    before skip=\medskipamount,
    after skip=\medskipamount
  ]
  \noindent
  \textbf{Problem~\theproblem}%
    \if\relax\detokenize{#1}\relax\else\ #1\fi%
  \itshape\ %
}%
{%
  \end{tcolorbox}%
  \endgroup
}

\newtheorem{theorem}{Theorem}[section]  
\newtheorem{lemma}[theorem]{Lemma}       

\newtheorem{assumption}{Assumption}[section]

\newtheorem{remark}[theorem]{Remark}

\begin{document}

\def\myConference{American Control Conference 2027}
\def\myCopyrightYear{2026}

\begin{figure*}[t]
  \thispagestyle{empty}
  \begin{center}
    \begin{minipage}{6in}
      \centering
      This paper has been submitted for presentation at the
      \emph{\myConference}.
      \vspace{1em}

      This is the author's version of an article that has, or will be,
      published in this conference. Changes were, or will be, made to this
      version by the publisher prior to publication.
      \vspace{17cm}

      \copyright\ \myCopyrightYear\ IEEE. Personal use of this material is
      permitted. Permission from IEEE must be obtained for all other uses, in
      any current or future media, including reprinting/republishing this
      material for advertising or promotional purposes, creating new
      collective works, for resale or redistribution to servers or lists, or
      reuse of any copyrighted component of this work in other works.
    \end{minipage}
  \end{center}
\end{figure*}
\newpage
\clearpage
\pagenumbering{arabic}

\maketitle

\begin{abstract}

We address robust trajectory-planning problems with state constraints and with control/state-dependent uncertain inputs and measurements. Uncertain inputs and measurements satisfy quadratic inequalities with rank-2 indefinite multipliers; measurements depend on the uncertain-state realization and on the uncertain inputs. By means of Lagrangian duality, we derive finite nonsmooth constraints that guarantee robust satisfaction of the state constraints, yielding a formulation that can be addressed without bilevel solvers. We develop a two-step convex heuristic for the resulting nonconvex nonsmooth formulation, using a convex relaxation followed by a convex feasibility-recovery step. We test the two-step approach on a rendezvous scenario with state-dependent navigation uncertainty, and we reduce by a factor 13.8 the computation time with respect to the commercial solver Gurobi, while maintaining robust satisfaction of constraints.

\end{abstract}

\section{Introduction}
Robust control enables the design of a single control solution guaranteeing stability and performance for all uncertainties belonging to a prescribed set. In this paper, we consider systems in which this set depends on the realized state: an uncertain input changes the subsequent state, which in turn changes the admissible uncertainty at later stages. Quadratic inequalities (QIs) provide a natural framework for describing such dependencies and connecting them to convex optimization methods for analysis and synthesis. Each QI is represented by a symmetric multiplier matrix, which characterizes the quadratic form of the QI. QIs are widely used in a broad range of applications. For instance, incremental nonlinear dynamics around an equilibrium trajectory can be modeled with QIs \cite{Acikmese2002-cr}: Incremental Quadratic Constraints ($\delta$QCs) characterize how increments in state-dependent variables map to increments in the nonlinearities. This enables the synthesis of observer-based controllers \cite{Xu2021-bv} and state-feedback controllers \cite{Reynolds2021-vl} by solving a convex feasibility problem with linear matrix inequalities. The procedure can be applied as well to continuous time \cite{Kim2025-rb}. Stability of linear systems in feedback with output-dependent uncertainties is further ensured by satisfaction of certain QIs when uncertainties and outputs satisfy Integral Quadratic Constraints (IQCs) \cite{Megretski1997-gi, Seiler2015-pj}. IQCs have a wide set of applications; some of them include stability verification with state constraints \cite{Fetzer2018-gq}, robust control of uncertain systems \cite{Biertumpfel2025-fj, Schwenkel2026-oq}, contraction analysis \cite{Shima2026-jf}, model predictive control \cite{Schwenkel2025-pa} and reachability analysis \cite{Yin2021-zt}. QIs further benefit neural network analysis and synthesis: they can be used for rapid safety verification of appropriate neural networks \cite{Fazlyab2022-xj}, thus allowing to treat neural network controllers within classical Lyapunov stability analysis \cite{Yin2022-og} frameworks. On the other hand, QIs can be used as constraints to regularize neural network controllers \cite{Junnarkar2025-tq}, or can also guide safe-by-design parameterizations of neural network layers \cite{Pauli2026-jm}.\\
Since multipliers are usually indefinite matrices, the uncertainty domain they describe is nonconvex; full rank, rank-2 indefinite multipliers describe nonconvex uncertainty sets that can be represented as finite unions of convex sets, thus connecting directly with \textit{disjunctive programming} \cite{Balas1998-ti}. There is an extensive literature, as confirmed by survey \cite{Trespalacios2014-uv}; studies mainly divide between those that seek exact relaxations \cite{Gusev2025-yl}, and those that seek tighter relaxations for the integer solver \cite{Peng2024-jy}. In the control field, lossless convexification exploits a provably exact relaxation to handle some classes of integer constraints; proposed for deterministic continuous-time systems \cite{Acikmese2007-cp, Echigo2023-lp} it has been extended recently to discrete-time systems \cite{Luo2025-dt, Luo2026-nf} and to uncertain systems with additive state-dependent uncertainties \cite{Sheridan2025-rf}.\\
Outside of the scope of QIs, additional approaches for robust planning involve System Level Synthesis \cite{Chen2024-ox}, which exploits a disturbance-feedback control parameterization to handle both parametric and additive uncertainties, tube model predictive control \cite{Rakovic2016-yt}, which parameterizes the set to which the state belongs at each time step, and constraint-tightening approaches \cite{Bujarbaruah2022-ux} that compute back-offs on the constraints' right-hand sides and synthesize only the nominal trajectory online. The previous approaches are not mutually exclusive and have been combined to handle nonlinear systems, but still they require an overapproximation of the reachable set, thus introducing conservatism \cite{Kohler2021-rc}. Additional works focus on state-dependent uncertainties; modelling such \textit{decision-dependent} uncertainties can reduce conservatism typical of robust optimization \cite{Lappas2018-xy}, at the cost of a generally $\mathcal{NP}$-hard problem \cite{Nohadani2018-zf}. Either assuming that uncertainties depend on nominal states \cite{Malyuta2019-yo,Sheridan2023-cr}, or applying the planning in recursive fashion \cite{Soloperto2018-gg} allow to bypass problem hardness, respectively by simplifying the uncertainty model or by using the uncertainty model only for one time-step. How the reachable set of the dynamical system is approximated, which constitutes a problem on its own \cite{Lew2025-em}, discriminates between conservatism and optimality for all the methods mentioned above.\\
In this paper, we handle scalar uncertainties characterized by rank-2 multipliers, involving state-dependent measurements computed at the uncertain states, with a focus on finite representability of the resulting problem. First, we formulate the robust semi-infinite programming (SIP) problem without assuming that bounds on uncertainties are evaluated at the reference trajectory: our model retains the dependence of uncertainty bounds on realized states. We then adopt Lagrangian duality to derive finite nonsmooth constraints that guarantee robust satisfaction of the state constraints, under a verifiable structural assumption on uncertainty feedthrough matrices. This yields a formulation that can be addressed without bilevel solvers. The proposed dualization step resembles the approach that authors of \cite{Chen2023-ge} apply to polyhedral uncertainty domains. Our method handles instead unions of convex domains, thus broadening the class of state-dependent uncertainties that can be taken into account. We further propose a two-step algorithm based on convex optimization to solve the nonsmooth problem and we finally compare its performance against the state-of-the-art solver Gurobi on a rendezvous scenario with navigation uncertainty. On the tested instances, our algorithm reduces computation time by a factor of 13.8 relative to the corresponding mixed-integer formulation solved using Gurobi, while obtaining matching objective values and guaranteeing robust satisfaction of the state constraints. We will extend our work to multipliers of arbitrary rank, but in the present work we focus on rank-two multipliers only.\\
In Sec. \ref{sec:statement} we state the problem of interest. Sec. \ref{sec:convex_reduction} contains the major theorems: we derive finite nonsmooth constraints that guarantee robust satisfaction of the state constraints via Lagrangian duality (Sec. \ref{sec:fromSIP_tofinite}) and propose a two-step convex approach to address the resulting formulation (Sec. \ref{sec:convexification}). We describe the resulting algorithm, propose a numerical example and draw conclusions, respectively, in Sec. \ref{sec:algorithm}, \ref{sec:numerical_example} and \ref{sec:conclusions}.

\subsection{Notation}
The index set gathering indices from $1$ to $N$ is indicated as $\mathcal{I}^{N}$; the index set overloaded with the zero index is denoted as $\mathcal{I}^{N\cup 0}$. Given two sets $\mathcal{A}$,$\mathcal{B}$, we denote by $\mathcal{A}\times\mathcal{B}$ their Cartesian product. We indicate with $\bm{1}_n$ the vector of $1$'s of dimension $n$ and with $\mathbb{I}_n$ the identity matrix of dimension $n\times n$. We use the notation $\phi:\mathbb{R}\rightrightarrows\mathbb{R}$ to denote a set-valued mapping from $\mathbb{R}$ to $\mathbb{R}$.

\section{Problem statement}
\label{sec:statement}
Consider a controlled uncertain discrete-time linear system over an $N+1$-step horizon.
System dynamics are perturbed by a fixed number $n_c\geq1$ of scalar uncertain signals at every time step $k\in\mathcal{I}^{N\cup0}$. We refer to each signal as a \textit{channel}. For $j\in\mathcal{I}^{n_c}$, each channel is characterized by a scalar system measurement $q_k^{j}\in\mathbb{R}$ and scalar system input $p_k^{j}\in\mathbb{R}$. Define the channel vectors and the product mapping
\begin{equation*}
\begin{aligned}
\bm{p}_k&\coloneqq[p_k^{1},\ldots,p_k^{n_c}]^\top,\\
\bm{q}_k&\coloneqq[q_k^{1},\ldots,q_k^{n_c}]^\top,\\
\bm{\phi}_k(\bm{q}_k)&\coloneqq
\phi_k^{1}(q_k^{1})\times\cdots\times\phi_k^{n_c}(q_k^{n_c}),
\end{aligned}
\end{equation*}
where the set-valued map $\phi_k^{j}:\mathbb{R}\rightrightarrows\mathbb{R}$ describes one scalar channel.
The set-valued map $\phi_k^{j}$ is associated with the symmetric and indefinite \textit{multiplier matrix} $M_k^{j}\in\mathbb{R}^{2\times2}$. Therefore, by definition, $p_k^j \in \phi_k^{j}(q_k^j)$ if and only if $z_k^{j}\coloneqq[q_k^{j},p_k^{j}]^\top$ satisfies the following quadratic inequality.
\begin{equation}
(z_k^{j})^\top M_k^{j}z_k^{j}\geq0,\quad k\in\mathcal{I}^{N\cup0},\quad j\in\mathcal{I}^{n_c}.
\label{eq:quadratic_inequality_step}
\end{equation}
We define $\mathcal{M}_k^j$ as the set of $z_k^j$ that satisfy \eqref{eq:quadratic_inequality_step}.

The discrete-time system is described by the following dynamics
\begin{equation}
\begin{array}{l}
\left.\begin{array}{l}
x_{k+1} = A_k x_k + B_{u,k} u_k + B_{p,k}\bm{p}_k\\[1ex]
\bm{q}_k = C_k x_k + D_{u,k} u_k + D_{p,k}\bm{p}_k \\[1ex]
\bm{p}_k \in \bm{\phi}_k(\bm{q}_k)\end{array}
\right| \quad k \in \mathcal{I}^{N\cup0}
\end{array}
\label{eq:dyn_uncert}
\end{equation}
controlled by $N+1$ impulses $u_k$ along the control horizon $[t_0,t_f]$. Here $x_k\in\mathbb{R}^{n_x}$ defines the states, $u_k\in\mathbb{R}^{n_u}$ defines the impulsive control. Furthermore, $A_k\in\mathbb{R}^{n_x\times n_x}$ is the \textit{state matrix}, $B_{u,k}\in\mathbb{R}^{n_x\times n_u}$ is the \textit{input matrix}, $B_{p,k}\in\mathbb{R}^{n_x\times n_c}$ is the \textit{uncertainty matrix}, $C_k\in\mathbb{R}^{n_c\times n_x}$ is the \textit{measurement matrix}, $D_{u,k}\in\mathbb{R}^{n_c\times n_u}$ is the \textit{input feedthrough matrix} and $D_{p,k}\in\mathbb{R}^{n_c\times n_c}$ is the \textit{uncertainty feedthrough matrix}.

States $x_k$ and measurements $\bm{q}_k$ are uncertain for all $k$.

Define the \textit{stacked states, controls, uncertainties and measurements} as
\begin{equation*}
\adjustbox{max width=\linewidth}{$\begin{aligned}
\bm{x} &\coloneqq
\begin{bmatrix}
x_1\\
\vdots \\
x_N \\
x_{N+1}
\end{bmatrix},
&
\bm{u} &\coloneqq
\begin{bmatrix}
u_0\\
\vdots \\
u_N
\end{bmatrix},
& 
\bm{p} &\coloneqq
\begin{bmatrix}
\bm{p}_0\\
\vdots \\
\bm{p}_N
\end{bmatrix},
&
\bm{q} &\coloneqq
\begin{bmatrix}
\bm{q}_0\\
\vdots \\
\bm{q}_N
\end{bmatrix}.
\end{aligned}$}
\end{equation*}

Let $N_c\coloneqq n_c(N+1)$ denote the total number of scalar channels over the horizon, so $\bm{p},\bm{q}\in\mathbb{R}^{N_c}$. Define the following horizon mapping 
$$
\bm{\phi}(\bm{q})\coloneqq
\bm{\phi}_0(\bm{q}_0)\times\cdots\times\bm{\phi}_N(\bm{q}_N).
$$

The uncertain stacked states $\bm{x}$ and uncertain stacked measurements $\bm{q}$ are then given by
\begin{equation}
\begin{array}{l}
\bm{x} = \bm{A} x_0 + \bm{B}_{u} \bm{u} + \bm{B}_{p}\bm{p}\\[1ex]
\bm{q} = \bm{C} x_0 + \bm{D}_{u} \bm{u} + \bm{D}_{p}\bm{p} \\[1ex]
\bm{p} \in \bm{\phi}(\bm{q})
\end{array}
\label{eq:dyn_uncert_stacked}
\end{equation}

The mathematical definition of the \textit{stacked matrices} $\bm{A}$, $\bm{B}_u, \bm{B}_{p},$ $\bm{C}$, $\bm{D}_u$ and $\bm{D}_p$ is reported in Appendix \ref{sec:appendix} for completeness. We only highlight a feature of $\bm{D}_u$ and $\bm{D}_p$ in the next remark. 
\begin{remark}
    $\bm{D}_u, \bm{D}_p$ depend on both feedthrough matrices $D_{u,k}, D_{p,k}$ and measurement matrices $C_{k}$, hence $D_{u,k} = D_{p,k} = 0$ does not imply $\bm{D}_u = \bm{D}_p = 0$.
\end{remark}

State $\bm{x}$ is uncertain and there is an infinite number of possible state realizations. Our aim is to choose controls $\bm{u}$ and bounds $\beta\in\mathbb{R}^{n_\beta}$ such that, given $\alpha\in\mathbb{R}^{n_\beta\times n_x (N+1)}$, the linear state constraints
\begin{equation}
\alpha\bm{x} \leq \beta 
\label{eq:SIconstraint}
\end{equation}
are robustly satisfied, that is to say they are satisfied for all $\bm{x} \;\,\text{satisfying}\;\,\eqref{eq:dyn_uncert_stacked}$. 

Controls, constraint bounds, and the initial state are further restricted to prescribed polytopes, possibly singletons. Such sets capture actuation, design, and dispersion limits:
\begin{equation}
\bm{u}\in\mathcal{U}\subseteq\mathbb{R}^{n_u(N+1)},\quad \beta\in\mathcal{B}\subseteq\mathbb{R}^{n_\beta},\quad x_0\in\mathcal{X}_0\subseteq\mathbb{R}^{n_x}.
\label{eq:admissible_sets}
\end{equation}

Under robust state constraint satisfaction, we seek to minimize the convex objective $\mathcal{L}:\mathbb{R}^{n_u(N+1)}\times\mathbb{R}^{n_\beta}\rightarrow \mathbb{R}$, which depends on the nominal controls and on the bounds on state constraints $\beta$.

The decision variables shall satisfy an infinite number of state constraints, one for each realization of the uncertain state $\bm{x}$. The complete problem belongs then to the class of \textit{Semi-Infinite Programming} (SIP) problems. Using the \textit{robust counterpart}, the SIP problem can be written as a \textit{min-max} problem \cite{Ben-Tal2009-nv}. The \textit{decision variables} $x_0, \bm{u}, \beta$ are kept as outer decision variables; the variables $\bm{x}, \bm{p}$, characterizing the inner maximization, are instead referred to as \textit{inner decision variables}. The \textit{robust counterpart} is then outlined as follows, where the inner max-type constraint constitutes the \textit{robust constraint satisfaction} (RCS) condition. The maximum is applied element-wise to its vector argument.
\begin{problem}
\label{prob:robust_counterpart}
\label{prob:rcs}
\begin{equation}
\begin{array}{cl}
    \underset{x_0,\!\bm{u},\!\beta}{\min.} & \mathcal{L}(\bm{u}, \beta) \\[1ex]
    \mathrm{s.t.} & \bm{u}\in\mathcal{U},\ \beta\in\mathcal{B},\ x_0\in\mathcal{X}_0 \\[1ex]
                  & \setlength{\fboxsep}{1pt}\fbox{\renewcommand{\arraystretch}{1.1}\setlength{\arraycolsep}{1.5pt}
    $\begin{array}{rl}
    \multicolumn{2}{l}{\text{\normalfont\bfseries RCS}} \\
    \underset{\substack{\bm{x}\in\mathbb{R}^{n_x(N+1)}\\[-0.2ex] \bm{p}\in\mathcal{P}(x_0,\bm{u})}}{\max} & \alpha\bm{x} \leq \beta \\[1ex]
     \bm{x} = & \bm{A} x_0 + \bm{B}_{u} \bm{u} + \bm{B}_{p}\bm{p} \\[1ex]
    \mathcal{P}(x_0,\bm{u}) \coloneqq & \left\{\bm{p}\left|\begin{array}{l}
    \bm{q} = \bm{C} x_0 + \bm{D}_{u} \bm{u} + \bm{D}_{p}\bm{p} \\[1ex]
    \bm{p} \in \bm{\phi}(\bm{q})
    \end{array}\right.\right\}
    \end{array}$}
\end{array}
\label{eq:robust_counterpart_nested_rcs}
\end{equation}
\end{problem}

\section{Convex approach via Outer Factorization}
\label{sec:convex_reduction}
For notational simplicity, we begin by considering a single scalar uncertainty channel, with $p,q\in\mathbb{R}$ and multiplier matrix $M\in\mathbb{R}^{2\times2}$.
Let $V$ and $U$ gather the eigenvectors of $M$ scaled by the square roots of the absolute values of the respective positive/negative eigenvalues of $M$. Define the vectors $g, h$ as in the next equations
\begin{equation}
\begin{array}{l}
g = V + U\\[0.5ex]
h = V - U.
\end{array}
\label{eq:vectors_gh}
\end{equation}
The multiplier $M$ can be outer-factorized, i.e. the vectors $g,h$ are \textit{outer factors} of $M$, according to the following equation
$$
M = \dfrac{1}{2}(g h^T + h g^T).
$$
Rewriting 
$$
g = \left[\begin{array}{c}
    g_q \\
    g_p
\end{array}\right],\; h = \left[\begin{array}{c}
    h_q \\
    h_p
\end{array}\right],
$$
where $g_p$ and $h_p$ are the components of $g$ and $h$ corresponding to inputs $p$, and $g_q$ and $h_q$ are the components corresponding to measurements, we can define column vectors $E,F\in\mathbb{R}^{2\times1}$ as
$$
E \coloneqq \left[\begin{array}{c} g_p \\ h_p \end{array}\right], \qquad
F \coloneqq \left[\begin{array}{c} g_q \\ h_q \end{array}\right].
$$
The membership $z\in\mathcal{M}$ can be expressed in terms of matrices $E, F$. 
\begin{lemma}
    $z \in \mathcal{M}$ iff $E p + F q \geq 0$ or $E p + F q \leq 0$.
    \label{lem:mult_to_EF}
\end{lemma}
\begin{proof} The following set of conditions holds.
    $$
    \begin{array}{cl}
    z^{\top}M z\geq0&\Leftrightarrow z^{\top}(gh^{\top}+hg^{\top})z\geq0\\[1ex]
    &\Leftrightarrow(g^\top z)^{\top}(h^\top z)\geq0\\[1ex]
    &\Leftrightarrow( g^\top z\geq0 \wedge h^\top z\geq0 )\quad \\
    & \qquad \qquad \vee \quad ( g^\top z\leq0 \wedge h^\top z\leq0 ) \\[1ex]
    &\Leftrightarrow E p + F q \geq 0 \quad \vee \quad  E p + F q\leq 0. 
     \end{array}
    $$
\end{proof}
\subsection{Reduction of semi-infinite problem to mixed-integer problem}
\label{sec:fromSIP_tofinite}
The approach described in the previous paragraph extends naturally to the multi-channel case. For each time $k\in\mathcal{I}^{N\cup0}$ and channel $j\in\mathcal{I}^{n_c}$, let $E_k^{j},F_k^{j}$ be the coefficient vectors obtained as above for $M_k^{j}$. For ease of notation, unify $k$ and $j$ into a single index $c$ defined as $c(k,j) \coloneqq j + k\,n_c$, running from $1$ to the total number of scalar channels $N_c$. Furthermore, let the shorthand $\lidx{(\cdot)}{c}$ denote quantity $(\cdot)_{k(c)}^{j(c)}$, e.g. $\lidx{E}{c}\coloneqq E_{k(c)}^{j(c)}$.

For ease of analysis, let us restrict Problem \ref{prob:rcs} to a single scalar constraint imposed at final time. Each domain associated with each $\lidx{\phi}{c}$ is nonconvex, and can be represented exactly by a finite disjunction. Applying the outer factorization (OF) described in Lemma \ref{lem:mult_to_EF} to each $\lidx{\phi}{c}$, we obtain the following RCS condition.
{\renewcommand{\theproblem}{2-RCS-OF}
\begin{problem}
\label{prob:rcs_multistep_decomposed}
\begin{equation}
\adjustbox{max width=\linewidth}{$\begin{array}{cl}
    \underset{\substack{x_{N+1}\in\mathbb{R}^{n_x}\\[-0.2ex] \bm{p}\in\mathcal{P}(x_0,\bm{u})}}{\max} & \alpha x_{N+1} \leq \beta \\[1ex]
     \mathrm{s.t.}& x_{N+1} = \bm{A}_{N+1} x_0 + \bm{B}_{u,N+1} \bm{u} + \bm{B}_{p,N+1}\bm{p}\\[1ex]
    \mathcal{P}(x_0,\bm{u}) \coloneqq & \left\{\bm{p}\left|\begin{array}{l}
\bm{q} = \bm{C} x_0 + \bm{D}_{u} \bm{u} + \bm{D}_{p}\bm{p} \\[1ex]
\begin{array}{l}
\displaystyle \lidx{E}{c}\, \lidx{p}{c} + \lidx{F}{c}\, \lidx{q}{c} \geq 0 \; \vee \\[1ex]
\displaystyle \qquad \qquad \lidx{E}{c}\, \lidx{p}{c} + \lidx{F}{c}\, \lidx{q}{c}\leq 0 \\[1ex]
\qquad \qquad c\in\mathcal{I}^{N_c}
\end{array}
\end{array}
    \right.\right\}
\end{array}
$}
\label{eq:rcs_multistep_decomposed}
\end{equation}
\end{problem}
}
We first outline the disjunctions in Problem \ref{prob:rcs_multistep_decomposed} rearranging set $\mathcal{P}(x_0,\bm{u})$. For this purpose, we first remove the uncertain variable $\bm{q}$. Let the row vector $\lidx{e}{c}\in\mathbb{R}^{1\times N_c}$ select the component corresponding to index $c$ in the stacked vectors, so $\lidx{p}{c}=\lidx{e}{c}\bm{p}$ and $\lidx{q}{c}=\lidx{e}{c}\bm{q}$. $\lidx{q}{c}$, therefore, reads as follows
$$
\lidx{q}{c} = \lidx{e}{c}\left(\bm{C} x_0 + \bm{D}_u\bm{u} + \bm{D}_p\bm{p}\right).
$$
The set $\calP(x_0,\bm{u})$ can then be rearranged as in the next equation.
\begin{equation}
\adjustbox{max width=0.92\linewidth}{$\displaystyle
\mathcal{P}(x_0,\bm{u})=
\left\{\bm{p}\ \middle|\
\begin{array}{l}
\lidx{E}{c}\,\lidx{p}{c}+\lidx{F}{c}\,\lidx{e}{c}(\bm{C}x_0+\bm{D}_u\bm{u}+\bm{D}_p\bm{p})\geq0\\[0.5ex]
\multicolumn{1}{c}{\vee}\\[0.5ex]
\lidx{E}{c}\,\lidx{p}{c}+\lidx{F}{c}\,\lidx{e}{c}(\bm{C}x_0+\bm{D}_u\bm{u}+\bm{D}_p\bm{p})\leq0,\\[1ex]
\multicolumn{1}{r}{c\in\mathcal{I}^{N_c}}
\end{array}
\right\}.
$}
\label{eq:set_P_1}
\end{equation}
We reformulate \eqref{eq:set_P_1} to express $\calP(x_0,\bm{u})$ as a finite set of linear disjunctions. Introduce the function $\tilde{\beta}_N:\mathbb{R}^{n_x}\times\mathbb{R}^{n_u(N+1)}\times\mathbb{R}\rightarrow\mathbb{R}$ and the matrix $\lidx{G}{c} \in \mathbb{R}^{2\times N_c}$.
\begin{align}
    \tilde{\beta}_N(x_0,\bm{u},\beta)
    \coloneqq
    \beta - \alpha(\bm{A}_{N+1} x_0 + \bm{B}_{u,N+1}\bm{u})
    \label{eq:beta_tilde}
    \\
    \lidx{G}{c}
    \coloneqq
    \lidx{F}{c}\, \lidx{e}{c}\,\bm{D}_p
    +
    \begin{bmatrix}
        0 & \cdots & 0 & \lidx{E}{c} & 0 & \cdots & 0
    \end{bmatrix}
    \label{eq:G_matrix}
\end{align}
where the block $\lidx{E}{c}$ is placed in the column corresponding to $\lidx{p}{c}$. Furthermore, define function $\lidx{\zeta}{c}:\mathbb{R}^{n_x}\times\mathbb{R}^{n_u(N+1)}\rightarrow\mathbb{R}^{2}$ as follows.
\begin{equation}
    \lidx{\zeta}{c}(x_0,\bm{u})
    :=
    \lidx{F}{c}\, \lidx{e}{c}(\bm{C}x_0 + \bm{D}_u\bm{u}).
    \label{eq:zeta}
\end{equation}

For each $c\in\mathcal{I}^{N_c}$, define the set $\lidx{\mathcal{P}}{c}(x_0,\bm{u})$ as
\begin{align}
    \lidx{\mathcal{P}}{c}(x_0,\bm{u})
    &:=
    \left\{
    \bm{p} \ \middle|\
    \begin{aligned}
    &\lidx{\zeta}{c}(x_0,\bm{u})+\lidx{G}{c}\,\bm{p} \geq 0 \\
    &\quad \vee\; \lidx{\zeta}{c}(x_0,\bm{u})+\lidx{G}{c}\,\bm{p} \leq 0
    \end{aligned}
    \right\}
    \label{eq:set_Pk}
\end{align}
Plugging Eqs. \eqref{eq:G_matrix}, \eqref{eq:zeta} into \eqref{eq:set_P_1}, we obtain $\lidx{\mathcal{P}}{c}(\bm{u})$ for each $c\in\mathcal{I}^{N_c}$, i.e.
\begin{equation}
    \mathcal{P}(x_0,\bm{u}) = \bigcap_{c=1}^{N_c}\lidx{\mathcal{P}}{c}(x_0,\bm{u}).
    \label{eq:P_as_intersection}
\end{equation}
The original set $\mathcal{P}(x_0,\bm{u})$ can be then interpreted as intersection of all sets $\lidx{\mathcal{P}}{c}(x_0,\bm{u})$. Each $\lidx{\mathcal{P}}{c}(x_0,\bm{u})$ is a disjunction, then $\mathcal{P}(x_0,\bm{u})$ is itself a collection of $N_c$ disjunctions, each choosing between two systems of two linear inequalities; therefore it can be defined equivalently as a union of $n_{\texttt{P}} \coloneqq 2^{N_c}$ polyhedral sets. We formalize this next.
Define the stacked function $\bm{\zeta}(x_0,\bm{u}):\mathbb{R}^{n_x}\times\mathbb{R}^{n_u(N+1)}\rightarrow \mathbb{R}^{2N_c}$ and the stacked matrix $\boldsymbol{G} \in \mathbb{R}^{2N_c\times N_c}$:
$$
    \bm{\zeta}(x_0,\bm{u})
    :=
    \begin{bmatrix}
        \lidx{\zeta}{1}(x_0,\bm{u}) \\
        \vdots \\
        \lidx{\zeta}{N_c}(x_0,\bm{u})
    \end{bmatrix},
    \qquad
    \boldsymbol{G}
    :=
    \begin{bmatrix}
        \lidx{G}{1} \\
        \vdots \\
        \lidx{G}{N_c}
    \end{bmatrix}.
$$
Then, for each $\ell\in\mathcal{I}^{n_{\texttt{P}}}$ define the set $\texttt{P}_{\ell}\,(x_0,\bm{u})$ and matrix $\Se_{\ell}\in\mathbb{R}^{2N_c\times2N_c}$ as follows 
\begin{align}
    \texttt{P}_{\ell}\,(x_0,\bm{u})&\coloneqq \left\{\bm{p} \ \middle|\ \Se_{\ell}\left(\bm{\zeta}(x_0,\bm{u})+\boldsymbol{G} \bm{p}\right) \geq 0\right\} \\
    \Se_{\ell}
    &\in
    \left\{
    \begin{bmatrix}
    \lidx{I}{1} &        &        \\
        & \ddots &        \\
        &        & \lidx{I}{N_c}
    \end{bmatrix}
    :
    \lidx{I}{c} \in \{\mathbb{I}_{2},-\mathbb{I}_{2}\}
    \right\}
    \label{eq:Se_selector}
\end{align}
that is, each $\mathbf{S}_{\ell}$ selects one sign pattern among $n_{\texttt{P}}$
possible sign choices and premultiplies the block $\bm{\zeta}(x_0,\bm{u}) + \boldsymbol{G} \bm{p}$. The next lemma formalizes the equivalence between $\mathcal{P}(x_0,\bm{u})$ and a union of $n_{\texttt{P}} = 2^{N_c}$ polyhedral sets.

\begin{lemma}
    The following equivalence holds.
\label{lem:union_of_intersections}
\begin{equation}
\mathcal{P}(x_0,\bm{u})=\bigcup_{\ell=1}^{n_{\texttt{P}}}\texttt{P}_{\ell}\,(x_0,\bm{u}).
\label{eq:uncertainty_polyhedral_union}
\end{equation}
\end{lemma}
\begin{proof}
Drop for simplicity $(x_0,\bm u)$. For each $c\in\mathcal{I}^{N_c}$, let $\lidx{\mathcal{P}}{c}^{+}\coloneqq \{\bm{p}\,|\, \lidx{\zeta}{c}+\lidx{G}{c}\,\bm{p} \geq 0\}, \lidx{\mathcal{P}}{c}^{-}\coloneqq \{\bm{p}\,|\, \lidx{\zeta}{c}+\lidx{G}{c}\,\bm{p} \leq 0\}$. Thus
$\lidx{\mathcal{P}}{c}
=\lidx{\mathcal{P}}{c}^{+}\cup\lidx{\mathcal{P}}{c}^{-},$
which allows to rearrange \eqref{eq:P_as_intersection} as
$
\mathcal{P}
=
\bigcap_{c=1}^{N_c}
\left(
\lidx{\mathcal{P}}{c}^{+}\cup
\lidx{\mathcal{P}}{c}^{-}
\right).
$
By distributivity of intersection over union, it further holds
\begin{equation}
\bigcap_{c=1}^{N_c}
\left(
\lidx{\mathcal{P}}{c}^{+}\cup
\lidx{\mathcal{P}}{c}^{-}
\right)
=
\bigcup_{\sigma\in\{1,-1\}^{N_c}}
\bigcap_{c=1}^{N_c}
\lidx{\mathcal{P}}{c}^{\sigma_c}.
\label{eq:distributivity_union}
\end{equation}
Every vector of signs $\sigma=[\sigma_1,\ldots,\sigma_{N_c}]$ identifies a combination of polyhedra $\lidx{\mathcal{P}}{c}^{\sigma_c}, c=1,\dots, N_c$ among the possible $2^{N_c}$ total combinations. Let $\sigma_\ell$ be one of such $2^{N_c}$ sign vectors, matching the selector $\Se_\ell$ in \eqref{eq:Se_selector} such that if $\sigma_{\ell,c}=1$, then $\lidx{I}{c}=\mathbb{I}_2$, if $\sigma_{\ell,c}=-1$, then $\lidx{I}{c}=-\mathbb{I}_2$. Then, $\bigcap_{c=1}^{N_c}\lidx{\mathcal{P}}{c}^{\sigma_c}$ can be rewritten as follows
\begin{equation}
\bigcap_{c=1}^{N_c}\lidx{\mathcal{P}}{c}^{\sigma_c}
=
\texttt P_\ell(x_0,\bm u).
\label{eq:P_ell_intersection}
\end{equation}
$n_{\mathtt{P}}$ denotes the number of possible combinations; then, using \eqref{eq:distributivity_union} and \eqref{eq:P_ell_intersection}, $\mathcal P(x_0,\bm u)$ can be rewritten as
$
\mathcal P(x_0,\bm u) = \bigcup_{\sigma\in\{1,-1\}^{N_c}}
\bigcap_{c=1}^{N_c}
\lidx{\mathcal{P}}{c}^{\sigma_c}
=
\bigcup_{\ell=1}^{n_{\mathtt{P}}}\texttt P_\ell(x_0,\bm u).
$
\end{proof}

In Problem \ref{prob:rcs_multistep_decomposed}, we remove the inner decision variable $x_{N+1}$ by substituting the expression for $x_{N+1}$ into the maximization; furthermore, we use Eq. \eqref{eq:beta_tilde} to move the nominal state contribution to the right-hand side. The quantity to maximize is then $\alpha\bm{B}_{p,N+1}\bm{p}$. Replacing $\mathcal{P}(x_0,\bm{u})$ by its union of polyhedral sets in Eq. \eqref{eq:uncertainty_polyhedral_union} gives the following problem.
{\renewcommand{\theproblem}{2-RCS-OF - Disjunctive}
\begin{problem}
\label{prob:rcs_multistep_simplified}
\begin{equation}
\begin{array}{cl}
    \underset{\bm{p}}{\sup}
    &  \alpha \bm{B}_{p,N+1} \bm{p}
    \leq
    \tilde{\beta}_N(x_0,\bm{u},\beta)
    \\[1ex]
    \mathrm{s.t.}
    & \bm{p}
    \in
    \bigcup_{\ell=1}^{ n_{\texttt{P}}}\texttt{P}_{\ell}(x_0,\bm{u})
    .
\end{array}
\end{equation}
\end{problem}
}
Define the compact polytopes
$$
\begin{gathered}
\mathcal Y_\ell(\bm{D}_p)\coloneqq
\left\{y\in\mathbb R^{N_c}\;\middle|\;
\Se_\ell\boldsymbol{G}(\bm{D}_p) y\ge0,\ \|y\|_1\le1\right\},\\
\ell\in\mathcal I^{n_{\texttt P}},
\end{gathered}
$$
and denote their vertex sets by $V(\mathcal Y_\ell(\bm{D}_p))$.
\begin{assumption}
\label{ass:dual_feasibility_vertices}
For every $\ell\in\mathcal I^{n_{\texttt P}}$,
$$
\alpha\bm B_{p,N+1}\hat y\le0
\qquad\forall\hat y\in V(\mathcal Y_\ell(\bm{D}_p)).
$$
\end{assumption}
\par\addvspace{0.5\baselineskip}
\noindent
This condition can be checked offline, since $\boldsymbol G(\bm D_p)$ depends only on the prescribed matrix $\bm D_p$ for the fixed outer factors.
We are therefore ready to formulate the main lemma.
\begin{lemma}
\label{lem:decomposition_dual_multistep}
Under Assumption~\ref{ass:dual_feasibility_vertices}, the constraint in Problem \ref{prob:rcs_multistep_simplified} is satisfied if and only if
\begin{equation}
\begin{array}{rcl}
\underset{\ell\in\mathcal{I}^{n_{\texttt{P}}}}{\max}&\underset{\boldsymbol{\gamma}_{\ell}}{\inf}&
\boldsymbol{\gamma}_{\ell}^{\top}\Se_{\ell}\bm{\zeta}(x_0,\bm{u})\le\tilde{\beta}_N(x_0,\bm{u},\beta)\\[0.5ex]
& \mathrm{s.t.}& \bm{B}_{p,N+1}^{\top}\alpha^\top+(\Se_{\ell}\;\boldsymbol{G})^{\top}\boldsymbol{\gamma}_{\ell}=0,\quad \boldsymbol{\gamma}_{\ell}\ge0.
\end{array}
\label{eq:decomposition_dual_multistep}
\end{equation}
\end{lemma}
\vspace{0.5\baselineskip}
\begin{proof}
By convention, $\sup\varnothing=-\infty$.
The constraint in Problem \ref{prob:rcs_multistep_simplified} is equivalent to
\begin{equation}
    \underset{\ell\in\mathcal{I}^{n_{\texttt{P}}}}{\max}\sup_{\bm{p} \in \texttt{P}_{\ell}(x_0,\bm{u})}
    \; \alpha \bm{B}_{p,N+1} \bm{p}
    \leq
    \tilde{\beta}_N(x_0,\bm{u},\beta).
    \label{eq:constraint_simple_multiphase}
\end{equation}
Introduce a nonnegative multiplier $\bm\gamma_\ell\in\mathbb R^{2N_c}_+$ for each $\ell$.
By Assumption~\ref{ass:dual_feasibility_vertices}, the minimization problem dual to the inner maximization in \eqref{eq:constraint_simple_multiphase} is feasible
 for every $\ell$. Indeed, since $\mathcal Y_\ell(\bm{D}_p)$ is a compact polytope and the objective is linear, Assumption~\ref{ass:dual_feasibility_vertices} implies
\begin{equation}
\alpha\bm B_{p,N+1}y\le0
\qquad\forall y\in\mathcal Y_\ell(\bm{D}_p).
\label{eq:dual_feasibility_normalized}
\end{equation}
Now let $y\ne0$ satisfy $\Se_\ell\boldsymbol{G} y\ge0$.
Its normalization $y/\|y\|_1$ belongs to $\mathcal Y_\ell(\bm{D}_p)$; applying Eq.~\eqref{eq:dual_feasibility_normalized} to this normalized vector and multiplying by $\|y\|_1$ gives $\alpha\bm B_{p,N+1}y\le0$. Consequently, Eq.~\eqref{eq:dual_feasibility_normalized} implies
\begin{equation}
\Se_\ell\boldsymbol{G} y\ge0
\quad\Longrightarrow\quad
\alpha\bm B_{p,N+1}y\le0,
\label{eq:farkas_dual_feasibility}
\end{equation}
which holds trivially for $y=0$.
By Farkas' lemma \cite{Bertsimas1997-lo}, Eq.~\eqref{eq:farkas_dual_feasibility} is equivalent to the existence of $\bm\gamma_\ell\ge0$ satisfying
\begin{equation}
(\Se_\ell\boldsymbol{G})^\top\bm\gamma_\ell
=-\bm B_{p,N+1}^\top\alpha^\top.
\label{eq:dual_feasibility_certificate}
\end{equation}
Weak duality ensures that
\begin{equation}
\adjustbox{max width=\linewidth}{$
 \begin{array}{l}\displaystyle\sup_{\bm{p} \in \texttt{P}_{\ell}(x_0, \bm{u})}
    \; \alpha \bm{B}_{p,N+1} \bm{p} \\
    {}
 \end{array} \leq \begin{array}{l}
    \underset{\boldsymbol{\gamma}_{\ell}}{\inf}\;
    \boldsymbol{\gamma}_{\ell}^{\top}\Se_{\ell}\bm{\zeta}(x_0,\bm{u}) \\
\mathrm{s.t.} \;
\bm{B}_{p,N+1}^{\top}\alpha^\top+(\Se_{\ell}\;\boldsymbol{G})^{\top}\boldsymbol{\gamma}_{\ell}=0,\\
\phantom{\textrm{s.t.} \;} \;
\boldsymbol{\gamma}_{\ell}\ge0
    \end{array}
$}
\label{eq:branchwise_strong_duality}
\end{equation}
As established by Eq.~\eqref{eq:dual_feasibility_certificate}, the dual minimization problem is feasible; then inequality in \eqref{eq:branchwise_strong_duality} is tight over the extended real domain and strong duality holds. Plugging Eq.~\eqref{eq:branchwise_strong_duality} into Eq.~\eqref{eq:constraint_simple_multiphase} proves the claim.
\end{proof}
\begin{lemma}
\label{lem:G_full_column_rank}
The matrix $\boldsymbol{G}$ has full column rank.
\end{lemma}
\begin{proof}
Let $\bm{d}\in\ker(\boldsymbol{G})$. For each $c\in\mathcal{I}^{N_c}$,
Eq.~\eqref{eq:G_matrix} gives
$$
0=\lidx{G}{c}\,\bm{d}
=[\,\lidx{E}{c}\ \lidx{F}{c}\,]
\begin{bmatrix}
\lidx{e}{c}\,\bm{d}\\
\lidx{e}{c}\,\bm{D}_p\bm{d}
\end{bmatrix}.
$$
Since the outer factors of the indefinite matrix $M_{k(c)}^{j(c)}$ are linearly
independent, $[\,\lidx{E}{c}\ \lidx{F}{c}\,]$ is nonsingular. Thus
$\lidx{e}{c}\,\bm{d}=0$ for every $c$. The selectors $\lidx{e}{c}$ cover all
coordinates of $\bm{d}$, so $\bm{d}=0$.
\end{proof}

By Lemma~\ref{lem:G_full_column_rank}, $\dim\ker(\boldsymbol{G}^\top)=N_c$.
Let the columns of $\bm{N}\in\mathbb{R}^{2N_c\times N_c}$ form a basis of
this nullspace. Since $\boldsymbol{G}^\top$ is surjective, each solution
of the dual equality has the form
\begin{align*}
\boldsymbol{\gamma}_{\ell} &=
\Se_{\ell}(\bar{\boldsymbol{\gamma}} + \bm{N} \bm{v}_{\ell} ) \\
\bar{\boldsymbol{\gamma}} &\coloneqq - (\,\bm{G}^\top)^\dagger \bm{B}_{p,N+1}^{\top}\alpha^\top
\end{align*}
where $\bm{v}_{\ell} \in\mathbb{R}^{N_c}$ is a vector variable. Define the admissible polyhedral sets $\mathcal{V}_{\ell}, \ell\in\mathcal{I}^{n_{\texttt{P}}}$ as follows.
$$
\mathcal{V}_{\ell} \coloneqq \{ \bm{v} \ \mid \ \Se_\ell\, \bm{N} \bm{v} \geq - \Se_\ell \bar{\boldsymbol{\gamma}}\}.
$$
Then $\boldsymbol{\gamma}_{\ell}\ge0\Leftrightarrow\bm{v}_{\ell} \in \mathcal{V}_{\ell}$. Therefore $\bm{v}_{\ell}$ is a \emph{free variable} constrained to the polyhedron
$\mathcal{V}_{\ell}$. Denote its vertex set by $V(\mathcal{V}_{\ell})$ and its vertices by
$\hat{\bm{v}}^{(k)}_{\ell}\in V(\mathcal{V}_{\ell})$,
$k\in\mathcal{I}^{\texttt{v}_{\ell}}$, with
$\texttt{v}_{\ell}:=\lvert V(\mathcal{V}_{\ell})\rvert$. 

We are ready to formulate the main theorem. To do this, we introduce the helper function $\bm{\xi}(x_0,\bm{u}):\mathbb{R}^{n_x}\times\mathbb{R}^{n_u(N+1)}\rightarrow \mathbb{R}^{N_c}$ defined as
$\bm{\xi}(x_0,\bm{u})\coloneqq \boldsymbol{N}^{\top} \bm{\zeta}(x_0,\bm{u})$.
Define the function $\eta_N:\mathbb{R}^{n_x}\times\mathbb{R}^{n_u(N+1)}\times\mathbb{R}\rightarrow\mathbb{R}$ as follows.
$$
\eta_N(x_0,\bm{u},\beta)
\coloneqq\tilde{\beta}_N(x_0,\bm{u},\beta)
-\bm{\zeta}^\top(x_0,\bm{u})\bar{\bm{\gamma}}.
$$
\begin{theorem}
    Conditions in Lemma \ref{lem:decomposition_dual_multistep} are satisfied if function $\bm{\xi}(x_0,\bm{u})$ satisfies the following condition
    \begin{equation}
    \begin{array}{c}
        \displaystyle
    \underset{\ell\in\mathcal{I}^{n_{\texttt{P}}}}{\max}\underset{k\in\mathcal{I}^{\texttt{v}_{\ell}}}{\min}\bm{\xi}^{\top}(x_0,\bm{u}) \hat{\bm{v}}^{(k)}_{\ell}\leq \eta_N(x_0,\bm{u},\beta).
    \end{array}
    \label{eq:reformulated_maxmin}
\end{equation}
    \label{thm:closed_form_multistep}
\end{theorem}
\begin{proof}
    Drop dependency on $x_0, \bm{u}, \beta$ for simplicity. Each infimum in Eq. \eqref{eq:decomposition_dual_multistep} can then be rearranged as follows.
    \begin{align*}
    \underset{\boldsymbol{\gamma}_{\ell}}{\mathrm{inf}}\;
\boldsymbol{\gamma}_{\ell}^{\top}\Se_{\ell}\bm{\zeta} &= \underset{\boldsymbol{\gamma}_{\ell}}{\mathrm{inf}}\;
(\bar{\boldsymbol{\gamma}} + \bm{N} \bm{v}_{\ell} )^\top\Se^\top_{\ell}\Se_{\ell}\bm{\zeta} \\
&= \bar{\boldsymbol{\gamma}}^\top \bm{\zeta}  + \underset{\bm{v}_{\ell}}{\mathrm{inf}}\;
( \bm{N} \bm{v}_{\ell} )^\top\bm{\zeta} \\
&= \bm{\zeta}^\top\bar{\boldsymbol{\gamma}} + \underset{\bm{v}_{\ell}}{\mathrm{inf}}\;
\bm{\xi}^\top \bm{v}_{\ell}.
\end{align*}
    Therefore, each $\ell$-th infimum with its $\ell$-th equality constraint from Eq. \eqref{eq:decomposition_dual_multistep} can be simplified as follows.
\begin{equation}
\begin{array}{rl}
    \underset{v_{\ell}}{\mathrm{inf}}\;
    \bm{\xi}^{\top} v_{\ell}
    \le \tilde{\beta}_N - \bm{\zeta}^{\top} \bar{\boldsymbol{\gamma}}
    & \;\;v_{\ell} \in \mathcal{V}_{\ell}
    \quad \ell \in\mathcal{I}^{n_{\texttt{P}}}
\end{array}
\label{eq:intermediate_lemma4_theorem5}
\end{equation}
Fix $(x_0,\bm{u})$, hence $\bm{\xi}$, and let $\hat{\bm{v}}^{(k)}_{\ell}$ denote the $k$-th vertex of $\mathcal{V}_{\ell}$. The infimum of $\bm{\xi}^{\top} \bm{v}_{\ell}$, if finite, is obtained for $\bm{v}_{\ell}$ equal to one of the vertices of $\mathcal{V}_{\ell}$, by linear programming optimality. Eq. \eqref{eq:intermediate_lemma4_theorem5} can be then written as follows. 
\begin{equation}
\underset{k\in\mathcal{I}^{\texttt{v}_{\ell}}}{\min}\bm{\xi}^{\top}(x_0,\bm{u}) \hat{\bm{v}}^{(k)}_{\ell}\leq \eta_N(x_0,\bm{u},\beta)
\label{eq:final_lemma4_theorem5}
\end{equation}
The infimum of $\bm{\xi}^{\top} \bm{v}_{\ell}$, if not finite, is upper bounded by the minimum in \eqref{eq:final_lemma4_theorem5}.
Applying this reformulation for all indices $\ell$ in Eq. \eqref{eq:decomposition_dual_multistep} provides Eq. \eqref{eq:reformulated_maxmin}.
\end{proof}

Using Theorem \ref{thm:closed_form_multistep}, we ensure satisfaction of the condition in Prob. \ref{prob:rcs_multistep_simplified} by solving the following robust problem.
{\renewcommand{\theproblem}{2-OF - MaxMin}
\begin{problem}
\label{prob:innerproblem_multistep_integer}
\begin{equation}
\adjustbox{max width=\linewidth}{$\displaystyle
\begin{array}{cl}
    \underset{x_0,\!\bm{u},\!\beta}{\min.} & \mathcal{L}(\bm{u}, \beta) \\[1ex]
    \mathrm{s.t.} & \bm{u}\in\mathcal{U},\ \beta\in\mathcal{B},\ x_0\in\mathcal{X}_0 \\[1ex]
    &
\begin{array}{l}
     \underset{{\ell\in\mathcal{I}^{n_{\texttt{P}}}}}{\max}\underset{k\in\mathcal{I}^{\texttt{v}_{\ell}}}{\min} \bm{\xi}^\top(x_0,\bm{u}) \hat{\bm{v}}^{(k)}_{\ell} \leq \eta_N(x_0,\bm{u},\beta)
    \end{array}
\end{array}
$}
\label{eq:innerproblem_multistep_integer}
\end{equation}
\end{problem}}

We have started from an SIP problem and ensured satisfaction of its constraints by means of a finite nonconvex problem. This proposed reformulation allows to bypass the bilevel nature of the SIP problem and make the problem solvable without dedicated solvers. 

\subsection{Two-step convex approach}
\label{sec:convexification}

We here propose a two-step approach to solve Problem \ref{prob:innerproblem_multistep_integer}. We first solve a convex relaxation based on convex-hull relaxation \cite{Balas1979-du}, and we ensure robustness with a second convex step. Let $\bm{w}\coloneqq(x_0,\bm{u},\beta)$. Constraint \eqref{eq:reformulated_maxmin} can be then rewritten as follows 
$$
\underset{{\ell\in\mathcal{I}^{n_{\texttt{P}}}}}{\max}\underset{k\in\mathcal{I}^{\texttt{v}_{\ell}}}{\min} \bm{\xi}^\top(\bm{w}) \hat{\bm{v}}^{(k)}_{\ell} \leq \eta_N(\bm{w})
$$
which is equivalent to
\begin{equation}
\forall \ell \in \mathcal{I}^{n_{\texttt{P}}}\;\;
\boxed{\exists k \in\mathcal{I}^{\mathtt{v}_\ell} \quad \text{s.t.} \quad
\boldsymbol{\xi}^\top(\bm{w})\hat{\bm{v}}^{(k)}_{\ell} \leq \eta_N(\bm{w})}
\label{eq:bilinear}
\end{equation}
\subsubsection{Step 1 - Convex-Hull relaxation}
The boxed constraint in Eq. \eqref{eq:bilinear} is nonconvex; we relax it with a convex-hull formulation in the variable $\bm{w}$. Since $\mathcal{U}$, $\mathcal{B}$, $\mathcal{X}_0$ are polytopes, their Cartesian product
$\mathcal{W}\coloneqq\mathcal{X}_0\times\mathcal{U}\times\mathcal{B}=\{\bm{w}\mid \boldsymbol{A}_{\mathcal{W}}\bm{w}\le\boldsymbol{b}_{\mathcal{W}}\}$
is a polytope. For each $\ell\in\mathcal{I}^{n_{\texttt{P}}}$ and vertex $k\in\mathcal{I}^{\mathtt{v}_\ell}$, introduce the new variables $\bm{w}_{\ell,k}$ (copies of $\bm{w}$) and $\lambda_{\ell,k}\ge0$, stacked as
$$
\boldsymbol{W}_\ell\coloneqq\begin{bmatrix}\bm{w}_{\ell,1}&\cdots&\bm{w}_{\ell,\mathtt{v}_\ell}\end{bmatrix}, \qquad
\boldsymbol{\lambda}_\ell\coloneqq\begin{bmatrix}\lambda_{\ell,1}&\cdots&\lambda_{\ell,\mathtt{v}_\ell}\end{bmatrix}^{\mathsf T}.
$$
Collect the vertices as $\boldsymbol{V}_\ell\coloneqq\begin{bmatrix}\hat{\bm{v}}^{(1)}_\ell&\cdots&\hat{\bm{v}}^{(\mathtt{v}_\ell)}_\ell\end{bmatrix}$, and define the stacked maps
$$
\boldsymbol{\Xi}(\boldsymbol{W}_\ell)\coloneqq\begin{bmatrix}\boldsymbol{\xi}(\bm{w}_{\ell,1})&\cdots&\boldsymbol{\xi}(\bm{w}_{\ell,\mathtt{v}_\ell})\end{bmatrix},
$$
$$
\boldsymbol{H}_N(\boldsymbol{W}_\ell)\coloneqq\begin{bmatrix}\eta_N(\bm{w}_{\ell,1})\\ \vdots \\ \eta_N(\bm{w}_{\ell,\mathtt{v}_\ell})\end{bmatrix}.
$$
The new variables satisfy
$$
\bm{w}=\boldsymbol{W}_\ell\bm{1}, \qquad \bm{1}^{\mathsf T}\boldsymbol{\lambda}_\ell=1, \qquad \boldsymbol{\lambda}_\ell\ge0,
$$
$$
\boldsymbol{A}_{\mathcal{W}}\boldsymbol{W}_\ell\le\boldsymbol{b}_{\mathcal{W}}\boldsymbol{\lambda}_\ell^{\mathsf T}, \qquad
\operatorname{diag}\!\left(\boldsymbol{V}_\ell^{\mathsf T}\boldsymbol{\Xi}(\boldsymbol{W}_\ell)\right)\le\boldsymbol{H}_N(\boldsymbol{W}_\ell).
$$
Problem \ref{prob:innerproblem_multistep_integer} can therefore be convexified through the following convex-hull relaxation.
{\renewcommand{\theproblem}{2-OF - CH}
\begin{problem}
\label{prob:robust_counterpart_multistep_CH}
\begin{equation}
\adjustbox{max width=\linewidth}{$\displaystyle
\begin{array}{cl}
\underset{\bm{w},\,\{\boldsymbol{W}_\ell,\boldsymbol{\lambda}_\ell\}}{\min.}&
\mathcal{L}(\bm{w})\\[0.8ex]
\mathrm{s.t.}&
\left.\begin{array}{l}
\bm{w}=\boldsymbol{W}_\ell\bm{1}, \quad \bm{1}^{\mathsf T}\boldsymbol{\lambda}_\ell=1,\\[1.2ex]
\boldsymbol{A}_{\mathcal{W}}\boldsymbol{W}_\ell\le\boldsymbol{b}_{\mathcal{W}}\boldsymbol{\lambda}_\ell^{\mathsf T},\\[1ex]
\operatorname{diag}\!\left(\boldsymbol{V}_\ell^{\mathsf T}\boldsymbol{\Xi}(\boldsymbol{W}_\ell)\right)\le\boldsymbol{H}_N(\boldsymbol{W}_\ell),\\[1ex]
\boldsymbol{\lambda}_\ell\ge0
\end{array}\right| \;\; \ell \in \mathcal{I}^{n_{\texttt{P}}}.
\end{array}
$}
\label{eq:robust_counterpart_multistep_CH}
\end{equation}
\end{problem}}
\begin{remark}
Problem \ref{prob:robust_counterpart_multistep_CH} provides a lower bound on the optimal objective of Problem \ref{prob:innerproblem_multistep_integer}. If the state constraints, at the optimal solution, are active for dual variables belonging to $\mathcal{V}_{\ell}$ with a single vertex, the relaxation in Problem \ref{prob:robust_counterpart_multistep_CH} is lossless. Otherwise, its solution provides an efficient guess for one vertex of each polyhedron $\mathcal{V}_{\ell}$.
\end{remark}
\subsubsection{Step 2 - Solution Robustification}
Denote the solution of Problem \ref{prob:robust_counterpart_multistep_CH} by $\bm{w}^{*}$, and select for each $\ell\in\mathcal{I}^{n_{\texttt{P}}}$ the index $k_\ell^*\coloneqq
\underset{k\in\mathcal{I}^{\texttt{v}_{\ell}}}
{\operatorname{argmin}}\;
\boldsymbol{\xi}^{\top}(\bm{w}^{*})\hat{\bm{v}}_{\ell}^{(k)}$. Then we define
$$
\bm{v}_{\ell}^*\coloneqq\hat{\bm{v}}_{\ell}^{(k_\ell^*)},
\;
\bm{V}^*\coloneqq
\left[\begin{array}{ccc}
\bm{v}_{1}^{*} &
\dots &
\bm{v}_{n_{\texttt{P}}}^{*}
\end{array}\right].
$$
We then restore robustness by solving the following problem.

{\renewcommand{\theproblem}{2-OF - Robust}
\begin{problem}
\label{prob:robust_counterpart_multistep_robust_trace}
\begin{equation}
\begin{array}{cl}
\displaystyle\min_{\bm{w}} & \mathcal{L}(\bm{w})\\[0.8ex]
\mathrm{s.t.} &
\bm{w}\in\mathcal{W},\\[0.8ex]
&
\bm{\xi}^\top(\bm{w})\bm{V}^*
\leq \eta_N(\bm{w})\bm{1}.
\end{array}
\label{eq:robust_counterpart_multistep_robust_trace}
\end{equation}
\end{problem}}

\subsection{Complete algorithm}
\label{sec:algorithm}
We have derived finite nonsmooth constraints that guarantee robust satisfaction of the state constraints in Problem \ref{prob:robust_counterpart}. The proposed workflow is summarized in Algorithm \ref{alg:algorithm_scalar}, while the two-step approach is reported in Algorithm \ref{alg:two_pass_convex_hull}. In the development we have assumed a single constraint. For a vector constraint / constraint at multiple time steps we can use the same construction for each constraint and solve a single optimization problem, as done in the example in Sec \ref{sec:numerical_example}.

\begin{algorithm}[h!]
\caption{Robust Optimization with Nonconvex Uncertainty Sets}
\label{alg:algorithm_scalar}
\begin{algorithmic}[1]
\Require Parameter $N$; multipliers $\lidx{M}{c}$, variables $x_0,\bm{u},\bm{p}$
\State Assemble matrix $\bm{G}$ and function $\bm{\zeta}(x_0,\bm{u})$
\State Compute nullspace $\bm{N}$
\For{$\ell\in\mathcal{I}^{n_{\texttt{P}}}$}
    \State Compute the vertex set $\{\hat{\bm{v}}_{\ell}^{(k)}\}_{k\in\mathcal{I}^{\texttt{v}_{\ell}}}$
\EndFor
\State Run Algorithm~\ref{alg:two_pass_convex_hull}
\end{algorithmic}
\end{algorithm}

\begin{algorithm}[h!]
\caption{Two-step convex-hull relaxation}
\label{alg:two_pass_convex_hull}
\begin{algorithmic}[1]
\State $\bm{w}^{*}
\leftarrow \text{Solve~Problem~\ref{prob:robust_counterpart_multistep_CH}}$
\For{$\ell=1,\dots,n_{\texttt{P}}$}
    \State $k_\ell^*\coloneqq
{\operatorname{argmin}}_{k\in\mathcal{I}^{\texttt{v}_{\ell}}}\;
\boldsymbol{\xi}^{\top}(\bm{w}^{*})\hat{\bm{v}}_{\ell}^{(k)}$
    \State $\bm{v}_\ell^*\leftarrow\hat{\bm{v}}_\ell^{(k_\ell^*)}$
\EndFor
\State $\bm{V}^*\leftarrow
\left[\begin{array}{ccc}
\bm{v}_{1}^{*} &
\dots &
\bm{v}_{n_{\texttt{P}}}^{*}
\end{array}\right]$
\State $\bm{w}^*\leftarrow
\text{Solve~Problem~\ref{prob:robust_counterpart_multistep_robust_trace}}$
\end{algorithmic}
\end{algorithm}

\section{Numerical Example}
\label{sec:numerical_example}

We test the proposed scheme on a planar transfer rendezvous with Clohessy-Wiltshire dynamics, with initial position dispersion and uncertain feedback control due to state-dependent navigation errors. We report the problem construction in Appendix \ref{sec:appendix}, and we summarize its formulation in Problem \ref{prob:robust_cw_feedback}.
{\renewcommand{\theproblem}{3}
\begin{problem}[Robust Rendezvous]
\label{prob:robust_cw_feedback}
\begin{equation}
\adjustbox{max width=\linewidth}{$\displaystyle
\begin{array}{cl}
\underset{\tilde{x}_0,\bm{u},\beta}{\min.}&\mathcal{L}(\bm{u})\\[1ex]
\text{s.t.}&\left|\begin{array}{l}
\tilde{x}_0=\begin{bmatrix}x^\top_0, x^\top_0\end{bmatrix}^\top,\quad u_k\in\mathcal{U}_k,\; k\in\mathcal{I}^{2\cup0},\quad\beta=\bar{\beta}\\[1ex]
\tilde{x}_{k+1}=A_k\tilde{x}_k+B_{u,k} u_k+B_{p,k}\bm{p}_k,
\quad k\in\mathcal{I}^{2\cup0},\\[1ex]
\bm{q}_k=C_k\tilde{x}_k+D_{u,k}u_k,
\quad k\in\mathcal{I}^{2\cup0},\\[1ex]
\mathcal{P}(x_0,\bm{u})\coloneqq\left\{\bm{p}\;\middle|\;
p_k^j\in\phi(q_k^j),\quad k\in\mathcal{I}^{2\cup0}, j\in\mathcal{I}^{2}
\right\},\\[1.5ex]
\alpha r_3\le\beta
\qquad\forall\,\bm{p}\in\mathcal{P}(x_0,\bm{u}).
\end{array}\right.
\end{array}
$}
\label{eq:robust_cw_feedback}
\end{equation}
\end{problem}}
The problem can be treated as discussed in the paper after removing one empty uncertainty domain for each constraint $p_0^j\in\phi(q_0^j)$, as discussed in the appendix. We solve Problem \ref{prob:robust_cw_feedback} using the proposed custom algorithm and the second-order conic solver QOCO \cite{Chari2026-lv}. Furthermore, we solve Problem \ref{prob:robust_cw_feedback} by solving directly the formulation in Problem \ref{prob:innerproblem_multistep_integer} with the commercial solver Gurobi \cite{gurobi}.
Fig. \ref{fig:two_step_multistage_box} illustrates the working principle of the two-step algorithm: in blue the integrated trajectories, in green the convex hulls of all uncertainty realizations. The green box denotes the final keep-in zone. In the first step, in Fig. \ref{fig:step_one_CHW}, the solution is superoptimal; however, the solution provides a sensible guess for the optimal direction to find optimal vertices at. The second step, in Fig. \ref{fig:step_two_CHW}, restores robust feasibility for all executed trajectories.
\begin{figure}[t!]
\centering
\subfloat[Step 1, Problem~\ref{prob:robust_counterpart_multistep_CH}. \label{fig:step_one_CHW}]{
\includegraphics[width=0.98\linewidth]{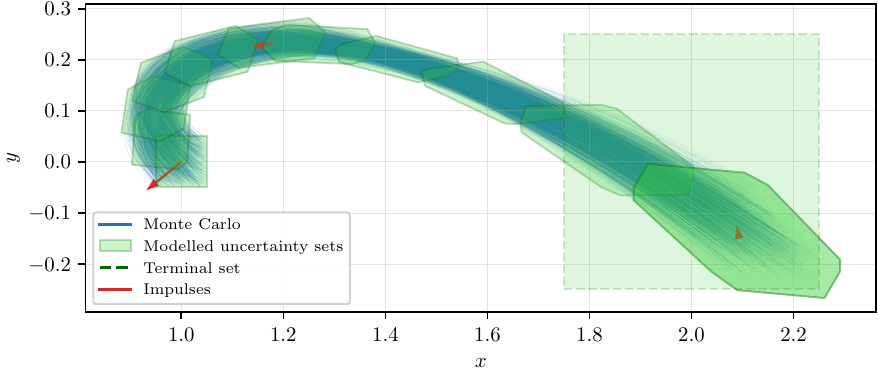}}\\
\subfloat[Step 2, Problem~\ref{prob:robust_counterpart_multistep_robust_trace}. \label{fig:step_two_CHW}]{
\includegraphics[width=0.98\linewidth]{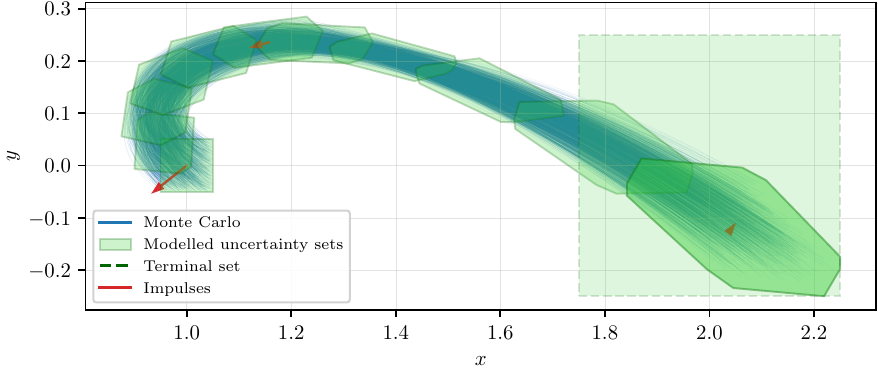}}
\caption{Comparison of the two-step solutions.}
\label{fig:two_step_multistage_box}
\end{figure}
In Tab. \ref{tab:solver_computational_times} we report the mean runtimes and standard deviations for Gurobi and QOCO over 30 different runs; these are obtained by uniformly rotating the final keep-in box with respect to its center by an angle between 0$^\circ$ and 90$^\circ$. For reference, the final keep-in box in Fig. \ref{fig:two_step_multistage_box} has an angle of 0$^\circ$. The Gurobi timings correspond to directly solving the mixed-integer max--min reformulation, Problem~\ref{prob:innerproblem_multistep_integer}, in place of Algorithm~\ref{alg:two_pass_convex_hull}; the QOCO timings instead accumulate the two convex solves within Algorithm~\ref{alg:two_pass_convex_hull} in QOCO first row and the remaining vertex-selection steps in QOCO second row. Algorithm 2 interface is implemented in Python, while QOCO is invoked directly via its API. For the scenario in Fig. \ref{fig:two_step_multistage_box}, the cost at convergence is $\mathcal{L}(\bm{u}) = 0.8741$, for both the proposed algorithm and Gurobi. The objectives of the proposed algorithm and Gurobi are the same, up to solver tolerance, for all 30 runs.
\begin{table}[h!]
\centering
\small
\begin{tabular}{llrr}
\toprule
\multicolumn{2}{c}{Solver} & Mean [ms] & Std. dev. [ms]\\
\midrule
\multicolumn{2}{c}{Gurobi} & 149.85 & 83.07\\
\cmidrule(lr){1-4}
\multirow{3}{*}{QOCO} & Solve & 7.42 & 0.44\\
 & Vertex selection & 3.46 & 0.11\\
 & Total & 10.88 & 0.44\\
\bottomrule
\end{tabular}
\caption{Computational times statistics over the run of 30 examples.}
\label{tab:solver_computational_times}
\end{table}

\section{Conclusion}
\label{sec:conclusions}
In this paper, we have tackled robust trajectory planning problems with uncertainties satisfying quadratic inequalities with rank-2 indefinite multipliers. We have shown how to cast a problem with uncertainties dependent on uncertain states as a bilevel program and then, by means of Lagrangian duality, we have derived finite nonsmooth constraints of $\textrm{max--min}$-type that guarantee robust satisfaction of the state constraints. We have proposed a heuristic, based on a convex relaxation followed by a convex feasibility-recovery step, for the nonsmooth problem, and tested it on a robust rendezvous problem with navigation uncertainty. We have reduced by a factor 13.8 the computation time with respect to the commercial solver Gurobi, while maintaining robust satisfaction of constraints.

\bibliographystyle{IEEEtran}
\bibliography{references}

@BOOK{Ben-Tal2009-nv,
  Title     = "{Robust Optimization}",
  author    = "Ben-Tal, Aharon and El Ghaoui, Laurent and Nemirovski, Arkadi",
  publisher = "Princeton University Press",
  address   = "Princeton, NJ",
  series    = "Princeton Series in Applied Mathematics",
  year      =  2009,
  language  = "en"
}

@ARTICLE{Chari2026-lv,
  Title     = "{QOCO: a Quadratic Objective Conic Optimizer with Custom Solver
               Generation}",
  author    = "Chari, Govind M and A\c{c}\i{}kme\c{s}e, Beh\c{c}et",
  journal   = "Mathematical Programming Computation",
  publisher = "Springer",
  month     =  "28~" # mar,
  year      =  2026,
  doi       = "10.1007/s12532-026-00311-8",
  language  = "en"
}

@ARTICLE{Kim2025-rb,
  title        = "{Continuous-Time Constrained Funnel Synthesis for
                  Incrementally Quadratic Nonlinear Systems}",
  author       = "Kim, Taewan and Luo, Dayou and A\c{c}\i{}kme\c{s}e, Beh\c{c}et",
  journal      = "arXiv [math.OC]",
  month        =  "12~" # nov,
  year         =  2025,
  primaryClass = "math.OC",
  doi          = "10.48550/arXiv.2511.08868"
}

@ARTICLE{Megretski1997-gi,
  Title     = "{System Analysis via Integral Quadratic Constraints}",
  author    = "Megretski, A and Rantzer, A",
  journal   = "IEEE Transactions on Automatic Control",
  publisher = "Institute of Electrical and Electronics Engineers (IEEE)",
  volume    =  42,
  number    =  6,
  pages     = "819--830",
  month     =  jun,
  year      =  1997,
  doi       = "10.1109/9.587335"
}

@ARTICLE{Luo2026-nf,
  Title     = "{Revisiting Lossless Convexification: Theoretical Guarantees for
               Discrete-Time Optimal Control Problems}",
  author    = "Luo, Dayou and Echigo, Kazuya and A\c{c}\i{}kme\c{s}e, Beh\c{c}et",
  journal   = "Automatica",
  publisher = "Elsevier BV",
  volume    =  183,
  month     =  jan,
  year      =  2026,
  doi       = "10.1016/j.automatica.2025.112537",
  language  = "en"
}

@ARTICLE{Acikmese2007-cp,
  Title     = "{Convex Programming Approach to Powered Descent Guidance for
               Mars Landing}",
  author    = "A\c{c}\i{}kme\c{s}e, Beh\c{c}et and Ploen, Scott R",
  journal   = "Journal of Guidance, Control, and Dynamics",
  publisher = "American Institute of Aeronautics and Astronautics (AIAA)",
  volume    =  30,
  number    =  5,
  pages     = "1353--1366",
  month     =  sep,
  year      =  2007,
  doi       = "10.2514/1.27553",
  language  = "en"
}

@ARTICLE{Echigo2023-lp,
  Title     = "{Linear Programming Approach to Relative-Orbit Control With
               Element-Wise Quantized Control}",
  author    = "Echigo, Kazuya and Hayner, Christopher R and Mittal, Avi and
               Sarsilmaz, Selahattin Burak and Harris, Matthew W and
               A\c{c}\i{}kme\c{s}e, Beh\c{c}et",
  journal   = "IEEE Control Systems Letters",
  publisher = "Institute of Electrical and Electronics Engineers (IEEE)",
  volume    =  7,
  pages     = "3042--3047",
  year      =  2023,
  doi       = "10.1109/LCSYS.2023.3289472",
  language  = "en"
}

@ARTICLE{Luo2025-dt,
  Title     = "{Discrete-Time Lossless Convexification for Pointing
               Constraints}",
  author    = "Luo, Dayou and Spada, Fabio and A\c{c}\i{}kme\c{s}e, Beh\c{c}et",
  journal   = "IEEE Control Systems Letters",
  publisher = "Institute of Electrical and Electronics Engineers (IEEE)",
  volume    =  9,
  pages     = "246--251",
  year      =  2025,
  doi       = "10.1109/LCSYS.2025.3569751",
  language  = "en"
}

@INPROCEEDINGS{Sheridan2023-cr,
  Title     = "{Convexification of Robust Trajectory Planning Problems With
               Nominal State and Control Dependent Uncertainties}",
  author    = "Sheridan, Oliver and A\c{c}\i{}kme\c{s}e, Beh\c{c}et",
  booktitle = "2023 62nd IEEE Conference on Decision and Control (CDC)",
  publisher = "IEEE",
  pages     = "6267--6272",
  year      =  2023,
  doi       = "10.1109/CDC49753.2023.10383510",
  language  = "en"
}

@ARTICLE{Sheridan2025-rf,
  Title     = "{Robust Fuel Optimal Trajectory Planning and Feedback Control
               for Constrained Linear Systems Under State- and Control-Dependent
               Perturbations}",
  author    = "Sheridan, Oliver and A\c{c}\i{}kme\c{s}e, Beh\c{c}et",
  journal   = "IEEE Control Systems Letters",
  publisher = "Institute of Electrical and Electronics Engineers (IEEE)",
  volume    =  9,
  pages     = "68--73",
  year      =  2025,
  doi       = "10.1109/LCSYS.2025.3558533",
  language  = "en"
}

@ARTICLE{Peng2024-jy,
  Title        = "{A Convexification-Based Outer-Approximation Method for Convex
                  and Nonconvex {MINLP}}",
  author       = "Peng, Zedong and Cao, Kaiyu and Furman, Kevin C and Li, Can
                  and Grossmann, Ignacio E and Bernal Neira, David E",
  journal      = "arXiv [math.OC]",
  month        =  "30~" # jul,
  year         =  2024,
  primaryClass = "math.OC",
  doi          = "10.48550/arXiv.2407.20973"
}

@ARTICLE{Gusev2025-yl,
  Title        = "{Exact Hull Reformulation for Quadratically Constrained
                  Generalized Disjunctive Programs}",
  author       = "Gusev, Sergey and Bernal Neira, David E",
  journal      = "arXiv [math.OC]",
  month        =  "22~" # aug,
  year         =  2025,
  primaryClass = "math.OC",
  doi          = "10.48550/arXiv.2508.16093",
  language     = "en"
}

@ARTICLE{Trespalacios2014-uv,
  Title     = "{Review of Mixed-Integer Nonlinear and Generalized Disjunctive
               Programming Methods}",
  author    = "Trespalacios, Francisco and Grossmann, Ignacio E",
  journal   = "Chemie Ingenieur Technik",
  publisher = "Wiley",
  volume    =  86,
  number    =  7,
  pages     = "991--1012",
  month     =  jul,
  year      =  2014,
  doi       = "10.1002/cite.201400037",
  language  = "en"
}

@ARTICLE{Balas1979-du,
  Title    = "{Disjunctive Programming}",
  author   = "Balas, Egon",
  journal  = "Annals of Discrete Mathematics",
  volume   =  5,
  pages    = "3--51",
  year     =  1979,
  language = "en"
}

@ARTICLE{Seiler2015-pj,
  Title     = "{Stability Analysis with Dissipation Inequalities and Integral
               Quadratic Constraints}",
  author    = "Seiler, Peter",
  journal   = "IEEE Transactions on Automatic Control",
  publisher = "Institute of Electrical and Electronics Engineers (IEEE)",
  volume    =  60,
  number    =  6,
  pages     = "1704--1709",
  month     =  jun,
  year      =  2015,
  doi       = "10.1109/tac.2014.2361004"
}

@INPROCEEDINGS{Schwenkel2025-pa,
  Title     = "{Output-Feedback Model Predictive Control under Dynamic
               Uncertainties Using Integral Quadratic Constraints}",
  author    = "Schwenkel, Lukas and K{\"{o}}hler, Johannes and M{\"{u}}ller,
               Matthias A and Allg{\"{o}}wer, Frank",
  booktitle = "{2025 IEEE 64th Conference on Decision and Control (CDC)}",
  publisher = "IEEE",
  pages     = "3514--3521",
  month     =  "9~" # dec,
  year      =  2025,
  doi       = "10.1109/cdc57313.2025.11312847"
}

@ARTICLE{Shima2026-jf,
  Title        = "{Regularized Model Predictive Control via Contractivity and
                  Implicit Lur'E Analysis}",
  author       = "Shima, Ryotaro and Gokhale, Anand and Davydov, Alexander and
                  Bullo, Francesco",
  journal      = "arXiv [math.OC]",
  month        =  "1~" # jul,
  year         =  2026,
  primaryClass = "math.OC",
  doi          = "10.48550/arXiv.2607.00383"
}

@INPROCEEDINGS{Biertumpfel2025-fj,
  Title     = "{Control Synthesis Along Uncertain Trajectories Using Integral
               Quadratic Constraints}",
  author    = "Biert{\"{u}}mpfel, Felix and Seiler, Peter and Pfifer, Harald",
  booktitle = "{2025 American Control Conference (ACC)}",
  publisher = "IEEE",
  pages     = "4357--4362",
  month     =  "8~" # jul,
  year      =  2025,
  doi       = "10.23919/acc63710.2025.11107524"
}

@ARTICLE{Fetzer2018-gq,
  Title     = "{Invariance with Dynamic Multipliers}",
  author    = "Fetzer, Matthias and Scherer, Carsten W and Veenman, Joost",
  journal   = "IEEE Transactions on Automatic Control",
  publisher = "Institute of Electrical and Electronics Engineers (IEEE)",
  volume    =  63,
  number    =  7,
  pages     = "1929--1942",
  month     =  jul,
  year      =  2018,
  doi       = "10.1109/tac.2017.2762764"
}

@ARTICLE{Yin2021-zt,
  Title     = "{Backward Reachability Using Integral Quadratic Constraints for
               Uncertain Nonlinear Systems}",
  author    = "Yin, He and Seiler, Peter and Arcak, Murat",
  journal   = "IEEE Control Systems Letters",
  publisher = "Institute of Electrical and Electronics Engineers (IEEE)",
  volume    =  5,
  number    =  2,
  pages     = "707--712",
  month     =  apr,
  year      =  2021,
  doi       = "10.1109/lcsys.2020.3005315"
}

@ARTICLE{Schwenkel2026-oq,
  Title     = "{Multi-Objective Robust Controller Synthesis with Integral
               Quadratic Constraints in Discrete-Time}",
  author    = "Schwenkel, Lukas and K{\"{o}}hler, Johannes and M{\"{u}}ller,
               Matthias A and Scherer, Carsten W and Allg{\"{o}}wer, Frank",
  journal   = "International Journal of Robust and Nonlinear Control",
  publisher = "Wiley",
  volume    =  36,
  number    =  3,
  pages     = "935--954",
  month     =  feb,
  year      =  2026,
  doi       = "10.1002/rnc.70098",
  language  = "en"
}

@ARTICLE{Fazlyab2022-xj,
  Title     = "{Safety Verification and Robustness Analysis of Neural Networks
               via Quadratic Constraints and Semidefinite Programming}",
  author    = "Fazlyab, Mahyar and Morari, Manfred and Pappas, George J",
  journal   = "IEEE Transactions on Automatic Control",
  publisher = "Institute of Electrical and Electronics Engineers (IEEE)",
  volume    =  67,
  number    =  1,
  pages     = "1--15",
  month     =  jan,
  year      =  2022,
  doi       = "10.1109/tac.2020.3046193"
}

@ARTICLE{Yin2022-og,
  Title     = "{Stability Analysis Using Quadratic Constraints for Systems with
               Neural Network Controllers}",
  author    = "Yin, He and Seiler, Peter and Arcak, Murat",
  journal   = "IEEE Transactions on Automatic Control",
  publisher = "Institute of Electrical and Electronics Engineers (IEEE)",
  volume    =  67,
  number    =  4,
  pages     = "1980--1987",
  month     =  apr,
  year      =  2022,
  doi       = "10.1109/tac.2021.3069388"
}

@INPROCEEDINGS{Junnarkar2025-tq,
  Title     = "{Stability Margins of Neural Network Controllers}",
  author    = "Junnarkar, Neelay and Arcak, Murat and Seiler, Peter",
  booktitle = "{2025 American Control Conference (ACC)}",
  publisher = "IEEE",
  pages     = "1355--1360",
  month     =  "8~" # jul,
  year      =  2025,
  doi       = "10.23919/acc63710.2025.11107746"
}

@ARTICLE{Pauli2026-jm,
  Title     = "{LipKernel: Lipschitz-Bounded Convolutional Neural Networks via
               Dissipative Layers}",
  author    = "Pauli, Patricia and Wang, Ruigang and Manchester, Ian R and
               Allg{\"{o}}wer, Frank",
  journal   = "Automatica",
  publisher = "Elsevier BV",
  volume    =  188,
  number    =  112959,
  pages     =  112959,
  month     =  jun,
  year      =  2026,
  doi       = "10.1016/j.automatica.2026.112959",
  language  = "en"
}

@article{Balas1998-ti,
  author  = {Balas, Egon},
  Title   = {Disjunctive Programming: Properties of the Convex Hull of Feasible Points},
  journal = {Discrete Applied Mathematics},
  volume  = {89},
  number  = {1--3},
  pages   = {3--44},
  year    = {1998},
  doi     = {10.1016/S0166-218X(98)00136-X}
}

@misc{gurobi,
  author = {{Gurobi Optimization, LLC}},
  Title = {{Gurobi Optimizer Reference Manual}},
  year = 2026,
  url = "https://www.gurobi.com"
}

@PHDTHESIS{Acikmese2002-cr,
  Title    = "{Stabilization, Observation, Tracking and Disturbance Rejection
              for Uncertain/Nonlinear and Time-Varying Systems}",
  author   = "A\c{c}\i{}kme\c{s}e, Ahmet Beh\c{c}et",
  address  = "West Lafayette, IN",
  year     =  2002,
  school   = "Purdue University",
  language = "en"
}

@ARTICLE{Xu2021-bv,
  Title     = "{Observer-Based Controllers for Incrementally Quadratic Nonlinear
               Systems with Disturbances}",
  author    = "Xu, Xiangru and A\c{c}\i{}kme\c{s}e, Beh\c{c}et and Corless, Martin",
  journal   = "IEEE Transactions on Automatic Control",
  publisher = "Institute of Electrical and Electronics Engineers (IEEE)",
  volume    =  66,
  number    =  3,
  pages     = "1129--1143",
  month     =  mar,
  year      =  2021,
  doi       = "10.1109/tac.2020.2996985"
}

@INPROCEEDINGS{Reynolds2021-vl,
  Title     = "{Funnel Synthesis for the 6-DOF Powered Descent Guidance Problem}",
  author    = "Reynolds, Taylor and Malyuta, Danylo and Mesbahi, Mehran and
               Acikmese, Behcet and Carson, John M",
  booktitle = "{AIAA Scitech 2021 Forum}",
  publisher = "American Institute of Aeronautics and Astronautics",
  address   = "Reston, Virginia",
  month     =  "11~" # jan,
  year      =  2021,
  doi       = "10.2514/6.2021-0504"
}

@ARTICLE{Chen2024-ox,
  title     = "{Robust Model Predictive Control with Polytopic Model Uncertainty
               through System Level Synthesis}",
  author    = "Chen, Shaoru and Preciado, Victor M and Morari, Manfred and
               Matni, Nikolai",
  journal   = "Automatica",
  publisher = "Elsevier BV",
  volume    =  162,
  number    =  111431,
  pages     =  111431,
  month     =  apr,
  year      =  2024,
  doi       = "10.1016/j.automatica.2023.111431",
  language  = "en"
}

@INPROCEEDINGS{Rakovic2016-yt,
  title     = "{Elastic Tube Model Predictive Control}",
  author    = "Rakovic, Sasa V and Levine, William S and Acikmese, Behcet",
  booktitle = "{2016 American Control Conference (ACC)}",
  publisher = "IEEE",
  pages     = "3594--3599",
  month     =  jul,
  year      =  2016,
  doi       = "10.1109/acc.2016.7525471"
}

@ARTICLE{Bujarbaruah2022-ux,
  title     = "{Robust MPC for LPV Systems via a Novel Optimization-based
               Constraint Tightening}",
  author    = "Bujarbaruah, Monimoy and Rosolia, Ugo and St{\"{u}}rz, Yvonne R
               and Zhang, Xiaojing and Borrelli, Francesco",
  journal   = "Automatica",
  publisher = "Elsevier BV",
  volume    =  143,
  number    =  110459,
  pages     =  110459,
  month     =  sep,
  year      =  2022,
  doi       = "10.1016/j.automatica.2022.110459",
  language  = "en"
}

@ARTICLE{Lew2025-em,
  title     = "{Convex Hulls of Reachable Sets}",
  author    = "Lew, Thomas and Bonalli, Riccardo and Pavone, Marco",
  journal   = "IEEE Transactions on Automatic Control",
  publisher = "IEEE",
  volume    =  70,
  number    =  12,
  pages     = "8195--8209",
  month     =  dec,
  year      =  2025,
  doi       = "10.1109/tac.2025.3586777"
}

@INPROCEEDINGS{Malyuta2019-yo,
  title     = "{Robust Model Predictive Control for Linear Systems with State-
               and Input-Dependent Uncertainties}",
  author    = "Malyuta, Danylo and Acikmese, Behcet and Cacan, Martin",
  booktitle = "{2019 American Control Conference (ACC)}",
  publisher = "IEEE",
  month     =  jul,
  year      =  2019,
  doi       = "10.23919/acc.2019.8815343"
}

@ARTICLE{Lappas2018-xy,
  title     = "{Robust Optimization for Decision-Making under Endogenous
               Uncertainty}",
  author    = "Lappas, Nikolaos H and Gounaris, Chrysanthos E",
  journal   = "Computers \& Chemical Engineering",
  publisher = "Elsevier BV",
  volume    =  111,
  pages     = "252--266",
  month     =  mar,
  year      =  2018,
  doi       = "10.1016/j.compchemeng.2018.01.006",
  language  = "en"
}

@ARTICLE{Nohadani2018-zf,
  title     = "{Optimization under Decision-Dependent Uncertainty}",
  author    = "Nohadani, Omid and Sharma, Kartikey",
  journal   = "SIAM Journal on Optimization",
  publisher = "Society for Industrial \& Applied Mathematics (SIAM)",
  volume    =  28,
  number    =  2,
  pages     = "1773--1795",
  month     =  jan,
  year      =  2018,
  doi       = "10.1137/17m1110560",
  language  = "en"
}

@ARTICLE{Soloperto2018-gg,
  title     = "{Learning-based robust model predictive control with
               state-dependent uncertainty}",
  author    = "Soloperto, Raffaele and M{\"{u}}ller, Matthias A and Trimpe,
               Sebastian and Allg{\"{o}}wer, Frank",
  journal   = "IFAC-PapersOnLine",
  publisher = "Elsevier BV",
  volume    =  51,
  number    =  20,
  pages     = "442--447",
  year      =  2018,
  doi       = "10.1016/j.ifacol.2018.11.052",
  language  = "en"
}

@BOOK{Bertsimas1997-lo,
  title     = "{Introduction to Linear Optimization}",
  author    = "Bertsimas, Dimitris and Tsitsiklis, John N.",
  publisher = "Athena Scientific",
  series    = "Athena Scientific Series in Optimization and Neural Computation",
  volume    =  6,
  year      =  1997,
  isbn      = "978-1-886529-19-9"
}

@ARTICLE{Chen2023-ge,
  title     = "{Robust Optimization with Continuous Decision-Dependent
               Uncertainty with Applications to Demand Response Management}",
  author    = "Chen, Hongfan and Sun, Xu Andy and Yang, Haoxiang",
  journal   = "SIAM Journal on Optimization",
  publisher = "Society for Industrial \& Applied Mathematics (SIAM)",
  volume    =  33,
  number    =  3,
  pages     = "2406--2434",
  year      =  2023,
  doi       = "10.1137/22m1502082",
  language  = "en"
}

@TECHREPORT{Blackford1968-orbital,
  title       = "{Guidance, Flight Mechanics and Trajectory Optimization. Volume XI---Guidance Equations for Orbital Operations}",
  author      = "Blackford, A. L. and Grier, D. R. and Townsend, G. E.",
  institution = "National Aeronautics and Space Administration",
  number      = "NASA-CR-1010",
  year        = 1968,
  month       = feb,
  url         = "https://ntrs.nasa.gov/citations/19680007746"
}

@ARTICLE{Kohler2021-rc,
  title     = "{A Computationally Efficient Robust Model Predictive Control
               Framework for Uncertain Nonlinear Systems}",
  author    = "Kohler, Johannes and Soloperto, Raffaele and Muller, Matthias A
               and Allgower, Frank",
  journal   = "IEEE Transactions on Automatic Control",
  publisher = "IEEE",
  volume    =  66,
  number    =  2,
  pages     = "794--801",
  month     =  feb,
  year      =  2021,
  doi       = "10.1109/tac.2020.2982585"
}

\section{Appendix}
\label{sec:appendix}
\subsection{Stacked matrices definitions}
The \textit{stacked matrices} $\bm{A}$, $\bm{B}_u$ and $\bm{B}_{p}$ are defined as follows.
$$
\bm{A} \coloneqq \left[\begin{array}{c}
    A_0 \\
    \vdots \\
    \prod_{j=0}^{N} A_j
\end{array}\right] \coloneqq \left[\begin{array}{c}
    \bm{A}_1 \\
    \vdots \\
    \bm{A}_{N+1}
\end{array}\right]$$
$$ \bm{B}_{u} \coloneqq \left[\begin{array}{ccc}
    B_{u,0} & \cdots & 0 \\
    \vdots & \ddots & \vdots \\
    \prod_{j=1}^{N} A_j B_{u,0} & \cdots & B_{u,N}
\end{array}\right] \coloneqq  \left[\begin{array}{c}
    \bm{B}_{u,1} \\
    \vdots \\
    \bm{B}_{u,N+1}
\end{array}\right]
$$
$$\bm{B}_{p} \coloneqq \left[\begin{array}{ccc}
    B_{p,0} & \cdots & 0 \\
    \vdots & \ddots & \vdots \\
    \prod_{j=1}^{N} A_j B_{p,0} & \cdots & B_{p,N} 
\end{array}\right]\coloneqq  \left[\begin{array}{c}
    \bm{B}_{p,1} \\
    \vdots \\
    \bm{B}_{p,N+1}
\end{array}\right]
$$
The \textit{stacked matrices} $\bm{C}$, $\bm{D}_u$ and $\bm{D}_p$ are defined as follows.
$$
\tilde{\bm{C}} \coloneqq \left[\begin{array}{ccccc} C_0 & 0 & \dots & 0  \\
    0 & \ddots & \ddots & \vdots \\ 
    \vdots & \ddots & \ddots & 0 \\ 
    0 & \dots & 0 & C_N 
\end{array}\right]; \quad 
\bm{C} \coloneqq \tilde{\bm{C}} \bm{A}^o \coloneqq  \left[\begin{array}{c}
    \bm{C}_{0} \\
    \bm{C}_{1} \\
    \vdots \\
    \bm{C}_{N}
\end{array}\right]
$$
$$
\bm{D}_{u} \coloneqq \tilde{\bm{C}} \bm{B}_u^o + \left[\begin{array}{cccc} D_{u,0} & 0 & \dots & 0 \\
    0 & \ddots & \ddots & \vdots \\ 
    \vdots & \ddots & \ddots & 0\\
    0 & \dots & 0 & D_{u,N}\end{array}\right] \coloneqq \left[\begin{array}{c}
    \bm{D}_{u,0} \\
    \bm{D}_{u,1} \\
    \vdots \\
    \bm{D}_{u,N}
\end{array}\right]
$$
$$\bm{D}_{p} \coloneqq \tilde{\bm{C}} \bm{B}_p^o + \left[\begin{array}{cccc} D_{p,0} & 0 & \dots & 0 \\
    0 & \ddots & \ddots & \vdots \\ 
    \vdots & \ddots & \ddots & 0 \\
    0 & \dots & 0 & D_{p,N}\end{array}\right]\coloneqq\left[\begin{array}{c}
    \bm{D}_{p,0} \\
    \bm{D}_{p,1} \\
    \vdots \\
    \bm{D}_{p,N}
\end{array}\right]
$$
where the \textit{stacked matrices with offset} $\bm{A}^{o}, \bm{B}_{u}^{o}, \bm{B}_{p}^{o}$ account for the time offset between $\bm{x}$ and $\bm{q}$ and are defined in the next equation.
$$
\bm{A}^{o} \coloneqq \left[\begin{array}{c}
    I\\
    \bm{A}_1 \\
    \vdots \\
    \bm{A}_N
\end{array}\right], \;
\bm{B}_{u}^{o} \coloneqq \left[\begin{array}{c}
    0\\
    \bm{B}_{u,1} \\
    \vdots \\
    \bm{B}_{u,N}
\end{array}\right], \;
\bm{B}_{p}^{o} \coloneqq \left[\begin{array}{c}
    0\\
    \bm{B}_{p,1} \\
    \vdots \\
    \bm{B}_{p,N}
\end{array}\right].
$$

\subsection{Numerical example description}
The nondimensional state is defined as $x=(r^\top,v^\top)^\top$, where $r$, $v$ are nondimensional position and velocity. Initial nominal state is given by $x_0$, and each component of the initial position dispersion satisfies a QI; both are provided as follows.
$$
x_0=\begin{bmatrix}1&0&0&1\end{bmatrix}^\top,
\qquad
|p_0^j|\le \bar{p}_{\max},\quad j\in\mathcal{I}^{2},
$$
while initial velocity is perfectly known. Uncertainty $p_0^j$ does not depend on state while it depends on the assigned constant $\bar{p}_{\max}$. The vector $\bm{p}_0=[p_0^1,p_0^2]^\top$ collects the two scalar initial position errors.
An initial and a midcourse impulse are used to control the system and to drive all dispersed trajectories to satisfy the final keep-in zone constraint.
$$
\alpha r_3\le\beta,\qquad \beta=\bar{\beta},
$$
where
$$
\alpha=\begin{bmatrix}I_2\\-I_2\end{bmatrix},\qquad
\bar{\beta}=\begin{bmatrix}r_f+h\bm{1}_2\\-r_f+h\bm{1}_2\end{bmatrix},
$$
$$
r_f=\begin{bmatrix}2\\0\end{bmatrix},\qquad h=0.25.
$$

The reference state $x^{\mathrm{ref}}_k$ augments the state $x_k$ to $\tilde{x}_k\coloneqq[(x_k)^\top,(x^{\mathrm{ref}}_k)^\top]^\top$. $x^{\mathrm{ref}}_k$ is obtained from the nominal feedforward control $\Delta v_k$, without any uncertainty. Each impulse $\Delta\tilde{v}_k$, $k\in\mathcal{I}^{2}$, in addition to the nominal feedforward component $\Delta v_k$ has a nominal feedback component $K_k(r_k-r^{\mathrm{ref}}_k)$ and uncertain feedback component $K_k\bm{p}_k$.
$$
\Delta\tilde{v}_k=\Delta v_k+K_k(r_k-r_k^{\mathrm{ref}})+K_k\bm{p}_k,
\qquad k\in\mathcal{I}^{2}.
$$

The nominal feedforward impulses satisfy $\|\Delta v_k\|_\infty\le \Delta v_{\max}$, $k\in\mathcal{I}^{2}$, where $\Delta v_{\max}=1$. This bound represents the finite velocity increment allocated for nominal maneuvers along each axis; the feedback corrections are modeled separately in $\Delta\tilde{v}_k$, and are allowed by some margin left with respect to the nominal feedforward authority.

$K_k$ is an assigned position-feedback gain, that we retrieve from a Fixed-Time-of-Arrival controller \cite{Blackford1968-orbital}.

The navigation error $\bm{p}_k$ depends on distance from a target placed at $\bar{r}_T=\begin{bmatrix}2.5&0\end{bmatrix}^\top$; the displacement with respect to target is computed as $\bm{q}_k=r_k - \bar{r}_T$. Each component $p_k^j$ and $q_k^j$ obey the following set-based dependency
$$
p_k^j\in\phi(q_k^j)\coloneqq
\left\{p\in\mathbb{R}\;\middle|\;|p|\le\gamma|q_k^j|\right\},
\quad k\in\mathcal{I}^{2},\ j\in\mathcal{I}^{2},
$$
with constant $\gamma=0.05$. Uncertainty $p_k^j$ is therefore state-dependent, and depends on the uncertain state $r_k$ and on the assigned constant $\bar{r}_T$. There are six independent scalar channels, collected in
$$
\bm{p}\coloneqq[\bm{p}_0^\top,\bm{p}_1^\top,\bm{p}_2^\top]^\top.
$$
Each scalar channel has its own multiplier. The initial position and navigation error multipliers are, respectively,
$$
M_0=\begin{bmatrix}1&0\\0&-1\end{bmatrix},\qquad
M_k=\begin{bmatrix}\gamma^2&0\\0&-1\end{bmatrix}.
$$
\begin{remark}
    While the navigation error satisfies a state-dependent QI, the initial position error satisfies a QI with a known term; the technique we have proposed for state-dependent QIs can be applied as well to state-independent QIs; however, since the QI relates uncertainty with an assigned value, one of the two linear cones stemming from the QI is infeasible and can be pruned in advance.
\end{remark}
An initial and a midcourse impulse are available, for a total of $N_1=2$ arcs of normalized duration $\Delta t_k=0.5$; however, since the initial position is uncertain, we introduce a first fictitious arc $N_0=1$ of duration $\Delta t_0=0$. The initial and midcourse controls are $\Delta v_1,\Delta v_2$, respectively; a terminal impulse does not affect the final position and is zero at the optimum; $\Delta v_0=0$ on the first fictitious arc. There are $n_c=2$ scalar channels at each of stages $k=0,1,2$, for a total of $N_c=6$.

Let
$$
A=
\begin{bmatrix}
0&0&1&0\\
0&0&0&1\\
3&0&0&2\\
0&0&-2&0
\end{bmatrix},
\; A^x_k\coloneqq e^{A\Delta t_k},\;
\begin{array}{l}
\Delta t_0=0,\\
\Delta t_1=0.5,\\
\Delta t_2=0.5.
\end{array}
$$
We can describe the system dynamics after defining the impulse matrix, position matrix and initial uncertainty matrix
$$
B=\begin{bmatrix}0_{2\times2}\\I_2\end{bmatrix},\qquad
C^r=\begin{bmatrix}I_2&0_{2\times2}\end{bmatrix},\qquad
J=\begin{bmatrix}I_2\\0_{2\times2}\end{bmatrix}.
$$
Define control and uncertainty matrices
$$
B^x_{\Delta v,0}=0_{4\times2},\qquad B^x_{p,0}=J,\qquad K_0=0_{2\times2},
$$
$$
B^x_{\Delta v,k}=A^x_kB,\qquad B^x_{p,k}=B^x_{\Delta v,k}K_k,
\quad k\in\mathcal{I}^{2}.
$$
Therefore, dynamics read as follows
\begin{align*}
\tilde{x}_{k+1} &=
\begin{bmatrix}
    \begin{array}{l}
A^x_k x_k+B^x_{\Delta v,k} \Delta v_k+\\
{}\qquad+B^x_{\Delta v,k} K_k(r_k-r_k^{\mathrm{ref}})+B^x_{p,k}\bm{p}_k
    \end{array}\\[0.5ex]
A^x_k x_k^{\mathrm{ref}}+B^x_{\Delta v,k} \Delta v_k
\end{bmatrix}\\[1ex]
&=A_k\tilde{x}_k+B_{\Delta v,k} \Delta v_k+B_{p,k}\bm{p}_k,
\quad k\in\mathcal{I}^{2\cup0},
\end{align*}
where
\begin{align*}
A_k&\coloneqq\begin{bmatrix}
A^x_k+B^x_{\Delta v,k}K_kC^r&-B^x_{\Delta v,k}K_kC^r\\
0&A^x_k
\end{bmatrix},\\
B_{\Delta v,k}&\coloneqq\begin{bmatrix}B^x_{\Delta v,k}\\B^x_{\Delta v,k}\end{bmatrix},
\qquad B_{p,k}\coloneqq\begin{bmatrix}B^x_{p,k}\\0_{4\times2}\end{bmatrix}.
\end{align*}
The initial state is
$$
\tilde{x}_0=\begin{bmatrix}x_0\\x_0\end{bmatrix}.
$$
The problem control $u_k$ embeds the impulse $\Delta v_k$, the target position $r_T$ and the initial position limit $p_{\text{max}}$ and is hence defined as 
$$
u_k \coloneqq \left[\Delta^\top v_k, r^\top_T, p_{\text{max}}\right]^\top,
$$
thus the control matrix reads
$$
B_{u,k}\coloneqq\begin{bmatrix}B_{\Delta v,k}&0_{8\times3}\end{bmatrix},
\qquad k\in\mathcal{I}^{2\cup0}.
$$
The admissible control sets $\mathcal{U}_k, \; k\in\mathcal{I}^{2\cup0}$ further read 
$$
\begin{array}{c}
    \mathcal{U}_0 \coloneqq \left\{ u \;\left| \begin{array}{l}
        \Delta v = 0 \\ r_T=\bar{r}_T \\ p_{\text{max}} = \bar{p}_{\text{max}}
\end{array}\right.\right\} \\[6ex]\mathcal{U}_k \coloneqq \left\{ u \;\left| \begin{array}{l}
        \|\Delta v\|_\infty \leq \Delta v_{\text{max}} \\ r_T=\bar{r}_T \\ p_{\text{max}} = \bar{p}_{\text{max}}
\end{array}\right.\right\},\quad k\in\mathcal{I}^2
\end{array}
$$
Furthermore, $\bm{q}_k$ is either fixed, as per the initial position uncertainty, or is affine in position, as per navigation error:
$$
\begin{array}{l}
    \bm{q}_0 = p_{\text{max}}\bm{1}_2 = D_{u,0}u_0\\[0.5ex]
    \bm{q}_k = r_k - r_{T} = C_k x_k + D_{u,k}u_k
\end{array}
$$
therefore
\begin{align*}
D_{u,0}&\coloneqq\begin{bmatrix}0_{2\times4}&\bm{1}_2\end{bmatrix},\\[0.5ex]
D_{u,k}&\coloneqq\begin{bmatrix}0_{2\times2}&-I_2&0_{2\times1}\end{bmatrix},
&&k\in\mathcal{I}^{2},
\end{align*}
while
$$
C_0\coloneqq0_{2\times8},\qquad
C_1=C_2\coloneqq\begin{bmatrix}C^r&0_{2\times4}\end{bmatrix}.
$$

We aim at minimizing the sum of squared control magnitudes, i.e.
$$
\mathcal{L}(\bm u)\coloneqq\sum_{k=1}^{2}\|\Delta v_k\|_2^2.
$$

\end{document}